\documentclass[12pt,english]{amsart}
\usepackage{amsmath}
\usepackage{libertine}
\usepackage[libertine]{newtxmath}
\usepackage{mathrsfs} 
\usepackage[T1]{fontenc}
\usepackage{hyperref}
\usepackage{url}
\usepackage{cite,amsmath,amstext,amsbsy,bm} 
\usepackage{paralist}
\usepackage{datetime} 
\usepackage{colortbl}
\usepackage[dvipsnames]{xcolor}
\usepackage{mathtools}
\usepackage{euscript}
\usepackage[all]{xy}

\usepackage{graphicx} 
\usepackage[text={150mm,251mm},centering]{geometry} 
\usepackage{comment}  
\usepackage{cite}  
\usepackage{thmtools}
\usepackage{enumitem}
\usepackage{letltxmacro} 
\usepackage{nameref}
\usepackage{cleveref}
\usepackage{calc}
\usepackage{interval}
\usepackage{manfnt}
\usepackage{tikz-cd}  
\usepackage[utf8]{inputenc}
\usepackage[hyperpageref]{backref}

\usepackage{csquotes}

\usepackage{faktor} 

\theoremstyle{plain}
\newtheorem{thmx}{Theorem}
\renewcommand{\thethmx}{\Alph{thmx}} 
\newtheorem{thm}{Theorem}[section]  
\newtheorem{lem}[thm]{Lemma}

\newtheorem{proposition}[thm]{Proposition}
\newtheorem{cor}[thm]{Corollary}

\newtheorem{claim}[thm]{Claim}

\theoremstyle{definition}
\newtheorem{dfn}[thm]{Definition}

\theoremstyle{remark}
\newtheorem{rem}[thm]{Remark}

\numberwithin{equation}{section}  

\theoremstyle{plain}
\newlist{thmlist}{enumerate}{1}
\setlist[thmlist]{wide = 0pt, labelwidth = 2em, labelsep*=0em, itemindent = 0pt, leftmargin = \dimexpr\labelwidth + \labelsep\relax, noitemsep,topsep = 1ex, font=\normalfont, label=(\roman*), ref=\thethm.(\roman{thmlisti})}

\addtotheorempostheadhook[thm]{\crefalias{thmlisti}{thm}}

\addtotheorempostheadhook[assumpsion]{\crefalias{thmlisti}{assumption}}

\addtotheorempostheadhook[cor]{\crefalias{thmlisti}{cor}}

\addtotheorempostheadhook[proposition]{\crefalias{thmlisti}{proposition}}

\addtotheorempostheadhook[dfn]{\crefalias{thmlisti}{dfn}}

\addtotheorempostheadhook[lem]{\crefalias{thmlisti}{lem}}
\addtotheorempostheadhook[main]{\crefalias{thmlisti}{main}}

\addtotheorempostheadhook[rem]{\crefalias{thmlisti}{rem}}

\newlist{thmenum}{enumerate}{1} 
\setlist[thmenum]{wide = 0pt, labelwidth = 2em, labelsep*=0em, itemindent = 0pt, leftmargin = \dimexpr\labelwidth + \labelsep\relax, noitemsep,topsep = 1ex, font=\normalfont, label=(\roman*), ref=\thethmx.(\roman{thmenumi})}
\crefalias{thmenumi}{thmx} 

\newlist{corlist}{enumerate}{1} 
\setlist[corlist]{wide = 0pt, labelwidth = 2em, labelsep*=0em, itemindent = 0pt, leftmargin = \dimexpr\labelwidth + \labelsep\relax, noitemsep,topsep = 1ex, font=\normalfont, label=(\roman*), ref=\thecorx.(\roman{corlisti})}
\crefalias{corlisti}{corx}

\makeatletter
\newsavebox{\@brx}
\newcommand{\llangle}[1][]{\savebox{\@brx}{\(\m@th{#1\langle}\)}%
	\mathopen{\copy\@brx\kern-0.5\wd\@brx\usebox{\@brx}}}
\newcommand{\rrangle}[1][]{\savebox{\@brx}{\(\m@th{#1\rangle}\)}%
	\mathclose{\copy\@brx\kern-0.5\wd\@brx\usebox{\@brx}}}
\makeatother

\crefname{lem}{Lemma}{Lemmas} 
\crefname{conjecture}{Conjecture}{Conjectures}
\crefname{thm}{Theorem}{Theorems}
\crefname{proposition}{Proposition}{Propositions}
\crefname{dfn}{Definition}{Definitions}
\crefname{rem}{Remark}{Remarks}
\crefname{cor}{Corollary}{Corollaries}
\crefname{corx}{Corollary}{Corollaries}
\crefname{problem}{Problem}{Problems}
\crefname{thmx}{Theorem}{Theorems}
\crefname{claim}{Claim}{Claims}
\crefname{assumption}{Assumption}{Assumptions}
\crefname{main}{Main Theorem}{Main Theorems}

\def\ep{\varepsilon}

\def\rank{{\rm rank}\,}

\makeatletter
\newcommand*{\rom}[1]{\expandafter\@slowromancap\romannumeral #1@}
\makeatother

\makeatletter     
\newcommand{\crefnames}[3]{%
	\@for\next:=#1\do{%
		\expandafter\crefname\expandafter{\next}{#2}{#3}%
	}%
}
\makeatother

\crefnames{part,chapter,section}{\S}{\S\S}

\newcommand{\cL}{\mathcal L}

\newcommand{\cO}{\mathcal O}

\newcommand{\bC}{\mathbb{C}}

\newcommand{\bP}{\mathbb{P}}

\newcommand{\bR}{\mathbb{R}}

\newcommand{\bZ}{\mathbb{Z}}

 \def\d{\partial} 
\def\hess{{\rm d}{\rm d}^{\rm c}}

\def\sn{\sqrt{-1}}

\def\End{{\rm \small  End}}

\newcommand{\shb}{(E,\theta,h)}
 
\makeatletter
\hypersetup{
	pdfauthor={\authors},
	pdftitle={\@title},
	pdfsubject={\@subjclass},
	pdfkeywords={\@keywords},
	pdfstartview={Fit},
	pdfpagelayout={TwoColumnRight},
	pdfpagemode={UseOutlines},
	bookmarks,
	colorlinks,
	linkcolor=Fuchsia,
	citecolor=NavyBlue,
	urlcolor=CarnationPink}
\makeatother

\begin{document} 
 	\title[Hyperbolicity of varieties admitting harmonic bundles]{Picard hyperbolicity of manifolds\\ admitting  
 		nilpotent  harmonic bundles}

	\author{Beno\^it Cadorel} 
\address{Institut \'Elie Cartan de Lorraine, Universit\'e de Lorraine, F-54000 Nancy,
	France.}	
\email{benoit.cadorel@univ-lorraine.fr}
	\urladdr{http://www.normalesup.org/~bcadorel/}
	\author{Ya Deng} 
 \address{CNRS, Institut \'Elie Cartan de Lorraine, Universit\'e de Lorraine, F-54000 Nancy,
 	France.}
 
	\email{ya.deng@univ-lorraine.fr \quad ya.deng@math.cnrs.fr}
	
	\urladdr{https://ydeng.perso.math.cnrs.fr} 
	\date{\today} 
	  
\begin{abstract} 
For a quasi-compact K\"ahler manifold $U$ endowed with a nilpotent harmonic bundle whose Higgs field is injective at one point, we prove that  $U$ is pseudo-algebraically hyperbolic, pseudo-Picard hyperbolic, and is of log general type. Moreover, we  prove that there is a finite unramified cover $\tilde{U}$ of $U$ from a quasi-projective manifold $\tilde{U}$ so that any projective compactification of $\tilde{U}$ is pseudo-algebraically hyperbolic, pseudo-Picard hyperbolic and is of general type.  As a byproduct, we establish some criterion of pseudo-Picard hyperbolicity and pseudo-algebraic hyperbolicity for quasi-compact K\"ahler manifolds. 
\end{abstract} 
	\maketitle
\tableofcontents	
	\section{Introduction} 
	\subsection{Main results}
	The notion of Picard hyperbolicity for quasi-compact K\"ahler manifolds, which was introduced in \cite{JK18,Den21}, is motivated by the classical big Picard theorem, which states that a holomorphic map \(\Delta^{\ast} \to \mathbb{P}^{1} \setminus \{0, 1 \infty\}\) extends as holomorphic map to the whole disk \(\Delta\). Complex manifolds sharing this property with \(\mathbb{P}^{1} \setminus\{0, 1, \infty\}\) are then said to be {\em Picard hyperbolic}. This notion turns out to be an important hyperbolicity property since it implies the algebraicity of analytic maps from quasi-projective manifolds to Picard hyperbolic ones; this was first proven in \cite{JK18}. The study of Picard hyperbolicity continues to have interesting developments: see e.g.\ the work of He--Ru \cite{HR21} where a quantitative version is introduced, or Etesse \cite{Ete20}, who introduces a notion of intermediate Picard hyperbolicity, and gives applications to finiteness properties of automorphism groups. 
	
	In \cite{Den21}, the second named author proved the Picard hyperbolicity for quasi-compact K\"ahler manifolds admitting a complex polarized variation of Hodge structures ($\bC$-PVHS for short) whose period map has zero dimensional fibers. $\bC$-PVHS is a subcategory of  \emph{nilpotent harmonic bundles}. Our goal of this paper is to extend the  results in \cite{Den21} to manifolds admitting   {nilpotent harmonic bundles}. The first result is the following.
	
	\begin{thmx}\label{main}
		Let $U$ be a quasi-compact K\"ahler manifold. Assume that there is a   nilpotent harmonic bundle $(E,\theta,h)$ over $U$ so that $\theta:T_{U}\to \End(E)$ is injective at one point.  Then $U$ is pseudo-Picard  hyperbolic, pseudo-algebraically hyperbolic, and is of log general type. Moreover, $U$ can be equipped with a \emph{unique} algebraic structure that makes it quasi-projective. 
	\end{thmx} 
See \cref{def:nilpotent,def:DC,def:Picard} for definitions of nilpotent harmonic bundles, pseudo-algebraic hyperbolicity and pseudo-Picard hyperbolicity.
Using ideas of \cite{Den21}, we can prove a stronger result on the hyperbolicity of compactifications after taking finite unramified cover of $U$, which is the main result of this paper.
 \begin{thmx}[$\subset$\cref{strong}]\label{main2}
	Let $U$ be a quasi-compact K\"ahler manifold.  Assume that    there is a nilpotent harmonic bundle $(E,\theta,h)$ on $U$ so that $\theta:T_{U}\to \End(E)$ is injective at one point.   Then there is a finite unramified cover $\tilde{U}\to U $ from a quasi-projective manifold $\tilde{U}$ so that any smooth projective compactification $X$ of $\tilde{U}$   is  of  general type, pseudo-algebraically hyperbolic and     pseudo-Picard hyperbolic. 
\end{thmx}

The proofs of \cref{main,main2} both rely on some new criterion of pseudo-Picard hyperbolicity and pseudo-algebraic hyperbolicity for  \emph{quasi-compact K\"ahler manifolds}, which is a novelty of this paper.
\begin{thmx}[$\subset$\cref{thm:criterion}+\cref{thm:algebraic}]\label{thmx:criterion}
	Let $Y$ be a compact K\"ahler manifold, and let $D$ be a simple normal crossing divisor on $Y$. Assume that $U:=Y-D$ is equipped with a pseudo-K\"ahler metric $\omega$  whose holomorphic sectional curvature   is bounded from above by a negative constant $-2\pi c$, then 
	\begin{thmlist}
		\item $U$ is pseudo-algebraically hyperbolic and pseudo-Picard hyperbolic.
		\item If the $(1,1)$ cohomology class $c\{\varpi\}-\{D\}$ is big, then $Y$ is pseudo-Picard hyperbolic and pseudo-algebraically hyperbolic. Here $\varpi$ is the closed positive $(1,1)$-current on $Y$ which is the trivial extension of $\omega$. 
	\end{thmlist}  
\end{thmx}

\subsection{Related works}
    After the work \cite{JK18},  Picard hyperbolicity  drew  a  lot  of  attention  over the last years. In \cite{BBT18},  the authors proved the algebraicity of analytic maps from a quasi-projective manifold to another one admitting a quasi-finite period map.  In \cite{DLSZ}  the second named author with Lu, Sun and Zuo proved the Picard hyperbolicity for moduli of polarized manifolds with semi-ample canonical sheaf. In \cite{Den21}, the second named author  proved \cref{main,main2} when the nilpotent harmonic bundle $(E,\theta,h)$ is moreover a complex polarized variation of Hodge structures. Similar results were later also obtained by Brotbek-Brunebarbe \cite{BB20} and Brunebarbe \cite{Bru20}.  
    Indeed, this paper is strongly inspired by the  works \cite{Den21,BB20,Bru20}: the proof of \cref{thm:criterion} is inspired by the Second Main Theorem of Brotbek-Brunebarbe \cite{BB20}, especially by their study of Nevanlinna characteristic function relative to positive currents; the proof of \cref{main2} follows  the same line as  that of \cite[Theorem B]{Den21}.  Though the methods in \cite{Den21} and \cite{BB20,Bru20} are quite different, a common ingredient is the use of Griffiths line bundle for systems of Hodge bundles, which has a nice global positivity property. For nilpotent harmonic bundles, we do not have such  analogous algebraic  objects,  and thus the methods and results in \cite{Den21,BB20} cannot be applied directly. A novelty in this paper is the use of a transcendental cohomology big class $\{\varpi\}$ (see \cref{strong} for the precise definition)  which plays a similar role in the proof as the Griffiths line bundle for complex variation of Hodge structures.  This also enables us to simplify previous work \cite{Den21,BB20,Bru20} since we do not use deep results in Hodge theory such as Schmid's nilpotent orbit theorems and  Hodge norm estimates etc.
    
    	In \cite{DLSZ} the second named author with Lu, Sun and Zuo obtained the criterion for Picard hyperbolicity in terms of a \emph{Finsler metric} for $T_Y(-\log D)$ with a \emph{stronger} curvature property than the negativity of holomorphic sectional curvature.  Indeed, we are not sure that our new criterion for Picard hyperbolicity \cref{thm:criterion} still holds if $\omega$ is only  assumed to a $(1,1)$-\emph{hermitian form}; its global positivity is crucial in the proof.    

	Theorem~\ref{strong} is a generalization of several earlier work, stemming from the seminal paper of Mumford \cite{Mum77}, who proved that given an arithmetic lattice on a bounded symmetric domain, then all compactifications of quotients by sublattices of sufficiently high index are of general type. It was then shown later by the work of Brunebarbe \cite{Bru16a}, Rousseau \cite{Rou16}, the first named author \cite{Cad18}   that these compactification satisfy very strong algebraic or hyperbolicity properties. These results were later extended to varieties supporting variations of Hodge structures in \cite[Theorem B]{Den21} and \cite{BB20,Bru20}; Theorem~\ref{strong} can be seen as a generalization of these last results for varieties supporting nilpotent harmonic bundles. 
    	
After the completion of this paper, Yohan Brunebarbe informed us that   they were able to prove that the quasi-projective manifold $U$ in \cref{main} is of   log general type  in an ongoing work with Jeremy Daniel  towards the Shafarevich conjecture for open varieties.

 \subsection{Acknowledgment}The second named author would like to thank Professors Takuro Mochizuki and Carlos Simpson, and Jeremy Daniel for very helpful discussions on  the proof of \cref{prop:extension}.  Both the authors would like to thank Yohan Brunebarbe and Damian Brotbek  for their remarks and interest on this work.

     \section*{Notations}
\begin{itemize}[leftmargin=0.3cm]
	\item A complex manifold is called \emph{quasi-compact K\"ahler} if it is the Zariski dense open set of a compact K\"ahler manifold. 
	\item For two real functions $f$ and $g$ on a complex manifold, we write $f\gtrsim g$ or $g\lesssim g$ if $f\geq \ep g$ for some constant $\ep >0$. 
	\item A \emph{compact K\"ahler log pair} (resp. \emph{projective log pair}) $(Y,D)$ consists of a compact K\"ahler (resp. \emph{projective}) manifold  $Y$ and a simple normal crossing divisor $D$ on $Y$. 
	\item A map $\mu:(X,\tilde{D})\to (Y,D)$ between compact K\"ahler log pairs is called a \emph{log morphism} if $\mu:X\to Y$ is a holomorphic map with $\tilde{D}\subset \mu^{-1}(D)$. 
	\item The unit disk is denoted by $\Delta$ and $\Delta^*$ denotes the punctured unit disk. 
	\item For any closed positive $(1,1)$ current $T$ on a compact K\"ahler manifold, we write $\{T\}$ for its cohomology class. For two cohomology $(1,1)$ class $\alpha$ and $\beta$, we write $\alpha\geq \beta$ if $\alpha-\beta$ is pseudo effective. 
	\item  For a line bundle $L$ with a singular hermitian metric $h$, its curvature current is denoted by $\Theta_h(L):=-\hess \log h$, where $\hess:=\frac{i}{2\pi}\d\bar{\d}$. 
\end{itemize} 
\section{Technical preliminaries} 
In this section we first recall the definitions of nilpotent   harmonic bundles and algebraic hyperbolicity. We then state and prove some results on Picard  hyperbolicity and closed positive $(1,1)$ currents, which will be used throughout this paper.
 \subsection{Harmonic bundles} \label{sec:VHS}  
  	\begin{dfn}[Higgs bundle]
 		A Higgs bundle on a complex manifold $Y$ is a pair $(E,\theta)$ consisting of a holomorphic vector bundle $E$ on $Y$ and an $\cO_Y$-linear map
 		$$
 		\theta:E\to E\otimes \Omega_Y^1
 		$$
 		so that $\theta\wedge\theta=0$. The map $\theta$ is called the \emph{Higgs field}.
 	\end{dfn} 
 \begin{dfn}[Harmonic bundle]
 	A \emph{harmonic bundle}  $(E,\theta,h)$ consists of a  Higgs bundle $(E,\theta)$ and  a hermitian metric $h$ for $E$ so that the connection
 	$$
 	D:=D_h+\theta+\theta_h^*
 	$$
 	is flat. Here $D_h$ is the Chern connection of $(E,h)$, and $\theta_h^*\in C^{\infty}(Y,\End(E)\otimes \Omega^{0,1}_Y)$ is the adjoint of $\theta$ with respect to $h$.
 \end{dfn}

\begin{dfn}[Nilpotent harmonic bundle]\label{def:nilpotent}
A harmonic bundle	  $\shb$ is called \emph{nilpotent}  if the characteristic polynomial $\det(t-\theta)= t^{\rank E}$.  
\end{dfn}  
Note that a complex polarized variation of Hodge structures induces a  nilpotent harmonic bundle.  

\subsection{Algebraic hyperbolicity}
Algebraic hyperbolicity for a compact complex manifold was introduced by Demailly in \cite[Definition 2.2]{Dem97}. He proved in \cite[Theorem 2.1]{Dem97} that a compact complex manifold is algebraically hyperbolic if it is  Kobayashi hyperbolic.  The notion of algebraic hyperbolicity was generalized to log pairs by Chen \cite{Che04}. 
\begin{dfn}[Pseudo-algebraic hyperbolicity]\label{def:DC}
	Let $(Y,\omega_Y)$ be a compact K\"ahler manifold and let $D$ be a simple normal crossing divisor on $Y$.  For  any reduced irreducible curve $C\subset Y$ such that $C\not\subset D$,  we denote by $i_Y(C,D)$ the number of distinct
	points in the set $\nu^{-1}(D)$, where $\nu:\tilde{C}\to C$
	is  the normalization of $C$. Assume that $Z\subsetneq Y$ is Zariski closed proper subset of $Y$. If   there is $\ep>0$ so that 
	$$
	2g(\tilde{C})-2+i_{Y}(C,D)\geq \ep{\rm deg}_{\omega_Y}C:=\ep \int_C\omega_Y
	$$
	for all reduced irreducible curve $C\subset Y$ not contained in $Z\cup D$,   $(Y,D)$ is called \emph{algebraically hyperbolic modulo $Z$}, and \emph{pseudo-algebraically hyperbolic}. If $Z=\varnothing$, $(Y,D)$ is called \emph{algebraically hyperbolic}.
\end{dfn} 
Note that the number $2g(\tilde{C})-2+i_{Y}(C,D)$ depends only on the intersection of \(C\) with the complement $Y-D$.
Hence the above notion of hyperbolicity also makes sense for quasi-projective manifolds: we say that a quasi-projective manifold $U$ is algebraically hyperbolic if it has a log compactification $(Y,D)$  which is  algebraically hyperbolic. 

However, it is unclear to us if Demailly's theorem extends to the non-compact case, i.e. if Kobayashi hyperbolicity, or Picard hyperbolicity, of $Y-D$ will imply the algebraic hyperbolicity of $(Y,D)$. Note that Pacienza-Rousseau \cite{PR07} have proved that if $Y-D$ is hyperbolically embedded into $Y$, the log pair $(Y,D)$ (and thus $Y-D$) is algebraically hyperbolic.

 \subsection{Picard hyperbolicity}
 Let us first recall the definition of Picard hyperbolicity introduced in \cite{Den21}. We start with the following definition of admissible coordinate systems which will be used frequently. 
 \begin{dfn}(Admissible coordinates)\label{def:admissible}  Let $Y$ be an $n$-dimensional complex manifold, and let $D$ be
 	a simple normal crossing divisor. Let $p$ be a point of $Y$, and assume that $\{D_{j}\}_{ j=1,\ldots,\ell}$ 
 	are the components of $D$ containing $p$. An \emph{admissible coordinate system} around $p$
 	is a tuple $(\Omega;z_1,\ldots,z_n;\varphi)$ (or simply  $(\Omega;z_1,\ldots,z_n)$ if no confusion arises) where
 	\begin{itemize}
 		\item $\Omega$ is an open subset of $Y$ containing $p$.
		\item \(\varphi\) is a holomorphic isomorphism $\varphi:\Omega\to \Delta^n$ so that  $\varphi(D_j)=(z_j=0)$ for any
 		$j=1,\ldots,\ell$.
 	\end{itemize}  
 \end{dfn}

 	\begin{dfn}[pseudo-Picard hyperbolicity]\label{def:Picard}
 	Let $U$ be a  quasi-compact K\"ahler manifold, and let $Y$ be a smooth K\"ahler  compactification.  $U$ is called \emph{pseudo-Picard hyperbolic} if there is a   Zariski closed proper subset $Z\subsetneq U$ so that any holomorphic map  $f:\Delta^*\to U$ with $f(\Delta^*)\not\subset Z$ extends to a holomorphic map $\bar{f}:\Delta\to Y$. We also say that $U$ is \emph{Picard hyperbolic modulo $Z$}. If $Z=\varnothing$, $U$ is simply called \emph{Picard hyperbolic}.
 \end{dfn}

In \cite[Lemma 4.3]{Den21} we proved that  \cref{def:Picard} does not depend on the compactification of $U$ when $Z=\varnothing$. The proof of this statement is based on the deep extension theorem of meromorphic maps by Siu \cite{Siu75}, and is also valid when $Z$ is not empty. Let us now give some interesting properties of pseudo-Picard hyperbolic manifolds. 

\begin{proposition} \label{extension theorem}
	Let $U$, $Y$ be as in \cref{def:Picard}, and assume that \(U\) is pseudo-Picard hyperbolic. Let $X$ be a compact complex manifold and let $D$ be a simple normal crossing divisor on $X$. If there is a holomorphic map $f:X-D\to U$ which is \emph{dominant}, then $f$ extends to a meromorphic map $X\dashrightarrow Y$. In particular, 
\begin{thmlist}
	\item\label{unique} any compact complex  manifold  containing $U$ as a Zariski dense open set  is bimeromorphic to $Y$.
	\item   the pseudo-Picard hyperbolicity of $U$ in \cref{def:Picard} does not depend on the compactification $Y$.
\end{thmlist}
\end{proposition}
\begin{proof} 
Write $V:=X-D$.  To prove that $f$ extends to a meromorphic map $X\dashrightarrow Y$, it suffices to check that locally around $D$. By \cite[Theorem 1]{Siu75}, any meromorphic map from a Zariski open set $W^\circ$ of a complex manifold $W$ to a compact K\"ahler manifold $Y$ extends to a meromorphic map from $W$ to $Y$ provided that the codimension of $W-W^\circ$ is at least 2.  It then suffices to consider the extensibility of $f$ around smooth points on $D$. Pick any such point $x\in D$ and choose admissible coordinates $(\Omega;x_1,\ldots,x_n)$ around $x$ so that $\Omega\cap D=(x_1=0)$. The theorem follows if we can prove that $f:\Delta^*\times \Delta^{n-1}\to U$ extends to a meromorphic map $\Delta^{n+1}\dashrightarrow Y$. Let $Z$ be the Zariski closed proper subset of $Y$ as in \cref{def:Picard}. Since $f$ is assumed to be dominant, there is thus a dense open set $W\subset \Delta^{n-1}$ so that for any $z\in W$, $f(\Delta^*\times \{z\})\not\subset Z$. Since $U$ is assumed to be Picard hyperbolic modulo $Z$, $f:\Delta^*\times \{z\}\to U$ extends to a holomorphic map $\Delta\times   \{z\}\to Y$ for $z\in W$.   It then follows from  \cite[p.442,  ($\ast$)]{Siu75} that $f$ extends to a meromorphic map  $\overline{f}:\Delta^{n}\dashrightarrow Y$. We thus can conclude that $f:X-D\to U$ extends to a meromorphic map $X\dashrightarrow Y$. 

Let $Y'$ be another compact complex manifold containing $U$ as a Zariski dense open set. We can  apply the Hironaka theorem on resolution of singularities  to assume that $Y'-U$ is a simple normal crossing divisor. By the above result, the identity map of $U$ extends to a meromorphic map $Y'\dashrightarrow Y$ which is thus bimeromorphic. The second statement follows, which also implies the last claim.
\end{proof}  
 \subsection{Closed positive $(1,1)$-currents}
 In this subsection we first  recall some results concerning closed positive $(1,1)$-currents (see \cite{Dem12}). We then prove  \cref{lem:crucial} which will be crucial in the proofs of our main results.
  \begin{dfn}[Pseudo-K\"ahler metric]
 	Let $X$ be a complex manifold. A $(1,1)$-form $\omega$ on $X$ is called a \emph{pseudo-K\"ahler metric} (or \emph{pseudo-K\"ahler form}) if $d\omega=0$, $\omega$ is semipositive, and strictly positive on a Zariski open set of $X$.
 \end{dfn}
 \begin{dfn}\label{def:bignef}
 	Let $(X,\omega)$ be a compact K\"ahler manifold. Let $\alpha\in H^{1,1}(X,\bR)$ be a cohomology $(1,1)$-class of $X$.   The class  $\alpha$ is  \emph{nef}  if for any $\ep>0$ there is a smooth closed $(1,1)$-form   $\eta_\ep\in \alpha$ so that $\eta_\ep\geq -\ep \omega$.   The class  $\alpha$ is  \emph{pseudo-effective}  if there is a closed positive $(1,1)$-current $T\in \alpha$. $\alpha$ is called \emph{big} if there is a closed positive $(1,1)$-current $T\in \alpha$ so that  $T\geq \delta\omega$ for some $\delta>0$. Such  a current $T$ will be called a \emph{K\"ahler current}.
 \end{dfn}

For two cohomology $(1,1)$ classes $\alpha$ and $\beta$, we write $\alpha\geq \beta$ if $\alpha-\beta$ is a pseudo-effective class.

Boucksom's criterion \cite{Bou02} asserts that a class $\alpha$ is big if there is a closed positive current $T\in \alpha$ so that $\int_{X}(T_{\rm ac})^{\rm dim X}>0$, where \(T_{\rm ac}\) denotes the  absolutely continuous part of \(T\) with respect to any smooth measure on \(X\). 

The \emph{non-K\"ahler locus} $E_{nK}(\alpha)$ of a big class $\alpha$  introduced  by Boucksom \cite{Bou04}  measures how far $\alpha$ is from being K\"ahler.  It is the transcendental generalization of the \emph{augmented base locus} for big line bundles.
\begin{dfn}[non-K\"ahler locus]\label{def:nK}
	Let $X$ be a compact K\"ahler manifold and let $\alpha$ be a big class on $X$. The \emph{non-K\"ahler locus} $E_{nK}(\alpha)$ of   $\alpha$ is 
	\begin{align*}
E_{nK}(\alpha):=\bigcap_{T\in \alpha}{\rm Sing}(T),
	\end{align*}
	where the intersection ranges over all K\"ahler currents $T\in \alpha$, and   ${\rm Sing}(T)$ is the complement of the set of points $x\in X$ such that $T$ is smooth around $x$.
\end{dfn}
Let us quote the following result by Boucksom \cite[Theorem 3.17]{Bou04}.
\begin{thm}[Boucksom]\label{boucksom}
Let $\alpha$ be a big class on a compact K\"ahler manifold. Then its {non-K\"ahler locus} $E_{nK}(\alpha)$ is a proper analytic subvariety. Moreover, there is a K\"ahler current $T\in \alpha$ with analytic singularities which is smooth outside $E_{nK}(\alpha)$.\qed
\end{thm}

If the class $\alpha$ is big and nef, in \cite{CT15}  Collins-Tosatti proved the following theorem on the characterization of  its non-K\"ahler locus  $E_{nK}(\alpha)$. It is a transcendental generalization of the Nakamaye  theorem.
\begin{thm}[Collins-Tosatti]\label{eq:CT}
Let $(X,\omega)$ be a compact K\"ahler manifold. Let $\alpha$ be a big and nef $(1,1)$ class on $X$. Then
\begin{align} 
	E_{nK}(\alpha)={\rm Null}(\alpha):=\bigcup_{\int_{Z}\alpha^{\dim Z}=0}Z
\end{align}
where  the union is taken over all positive dimensional irreducible analytic subvarieties $Z$ in $Y$. \qed
\end{thm}

Let us recall  the following extension theorem of Skoda (see \emph{e.g.} \cite[ (2.4) Theorem]{Dembook}) which will be used frequently.

\begin{thm}[Skoda]\label{lem:skoda}
Let $X$ be a  (not necessarily compact) complex manifold and let $A$ be a closed analytic subset of $X$. Assume that $T$ is a closed positive $(p,p)$-current defined on $X-A$  so that   $T$ has locally \emph{finite mass} in a neighborhood   of any point of $A$. Then the trivial extension of $T$, denoted by $\widetilde{T}$, is also a closed positive $(p,p)$-current on $X$.    \qed
\end{thm}

Recall that $\widetilde{T}$ is defined as follows. For any smooth test $(n-1,n-1)$-form $\eta$, we let
\begin{align}\label{eq:extension}
\widetilde{T}(\eta):=\int_{X-A}T\wedge \eta. 
\end{align}

In particular,  Skoda's theorem implies the following result due to Bishop \cite[Theorem 3]{Bis64}, which will be used to prove \cref{prop:criterion}.
\begin{thm}[Bishop]\label{thm:Bishop}
Let $X$ be a  (not necessarily compact) complex manifold and let $A$ be a closed analytic subset of $X$.   Let $V$ be an analytic subset of pure
dimension $p$ of $X-A$. Assume that $V$ has locally finite volume near $A$. Then the topological closure $\overline{V}$ of $V$ is
an analytic subset of $X$. \qed
\end{thm}

The following result, which is a variant of  in \cite[Lemma 5.4]{Den20},  will be crucial throughout this paper.
 \begin{lem}\label{lem:crucial}
 	Let $(Y,D)$ be a compact K\"ahler log pair.  Let $\omega$ be a pseudo-K\"ahler form on $U:=Y-D$ with   holomorphic sectional curvature bounded from above by a negative constant. Then
 	\begin{thmlist}
 		\item the trivial extension of $\omega$, denoted by $\varpi$, is a closed positive current;
 		\item  the cohomology class $\{\varpi\}$ is big and nef; 
 		\item for any admissible coordinates $(\Omega;z_1,\ldots,z_n)$, the local potential $\phi$ of $\varpi=\hess \phi$ satisfies
 	\begin{align}\label{eq:loglog}
 	\phi\gtrsim -\log \Big(\prod_{j=1}^{\ell}(-\log |z_j|^2)\Big).
 	\end{align} 
 	\end{thmlist}
 \end{lem}
\begin{proof}
Pick any point $x\in D$, and choose admissible coordinates $(\Omega;z_1,\ldots,z_n)$ centered at $x$ so that $D\cap \Omega=(z_1\cdots z_\ell=0)$. Since the  holomorphic sectional curvature of $\omega$  is bounded from above by a negative constant, by Ahlfors-Schwarz lemma, we can use \cite[Proposition 3.1.2]{Cad16}, which implies that there is a constant $\delta>0$ so that
\begin{align}\label{eq:bound}
 \frac{1}{\delta}\omega\leq  \omega_P:=\sum_{j=1}^{\ell}\frac{\sqrt{-1}dz_j\wedge d\bar{z}_j}{|z_j|^2(\log |z_j|^2)^2}+\sum_{k=\ell+1}^{n} \frac{\sqrt{-1}dz_k\wedge d\bar{z}_k  }{(1-|z_k|^2)^2}.
\end{align}
Consequently, the local mass of $\omega$ is bounded. By Skoda's theorem, its trivial extension $\varpi$ is a closed positive current.    Since $\int_{Y-D}\omega^n>0$, by Boucksom's criterion, $\{\varpi\}$ is big.

	Since \(\phi\) is a closed positive \((1, 1)\)-current, there is a psh function $\phi$ on $\Omega$ so that $\hess \phi=\varpi$. Now, by the very definition of trivial extension \eqref{eq:extension},  $ \varpi\leq \delta\omega_P$. Note that 
$$
\omega_P=-\hess \log \Big(\prod_{j=1}^{\ell}(-\log |z_j|^2)\cdot \prod_{k=\ell+1}^{n} (1-|z_k|^2) \Big).
$$
Since $\delta \omega_P-\varpi\geq 0$, the function 
$$
-\delta  \hess \log \Big(\prod_{j=1}^{\ell}(-\log |z_j|^2)\cdot \prod_{k=\ell+1}^{n} (1-|z_k|^2) \Big)-\phi
$$
	is thus a psh function, and as such, it is locally bounded from above. The inequality \eqref{eq:loglog} then follows.  Therefore, $\varpi$ has zero Lelong numbers everywhere. By the regularization theorem for closed positive currents of Demailly (see \cite[Corollary 6.4]{Dem92}), the class $\{\varpi\}$ is nef. The lemma is proved.
	\end{proof}

     The following result  due to Brunebarbe \cite[Proposition 3.3]{Bru20} will be used to prove \cref{general type}. For completeness sake, we provide a proof here.
   \begin{lem}\label{lem:pseudo}
   	Let $(Y,D)$ be a compact K\"ahler log-pair, and let $\omega$ be a pseudo-K\"ahler form on   $U:=Y-D$ so that it has non-positive holomorphic bisectional curvature and holomorphic sectional curvature bounded from above by a negative constant $-2\pi c$. Then the class $K_Y+D-c\{\varpi\}$ is pseudo-effective.
   \end{lem}
   \begin{proof}
   	Let $U_0$ be  the Zariski open set of $U$ so that $\omega$ is strictly positive definite. Since $\omega$ has non-positive holomorphic bisectional curvature and holomorphic sectional curvature bounded from above by a negative constant $-2\pi c$, one  has
   	$$
   	-{\rm Ric}(\omega)\geq  2\pi c\omega.
   	$$
   	Let $h_{K_Y+D}$ be the singular hermitian metric on $K_Y+D$ induced by $\omega$. We will prove that its curvature current $\Theta_{h_{K_Y+D}}(K_Y+D)$ is positive.

   	Pick any point $x\in D$, and choose admissible coordinates $(\Omega;z_1,\ldots,z_n)$ centered at $x$ so that $D\cap \Omega=(z_1\cdots z_\ell=0)$.  Then for the local frame $\sigma:=d\log z_1\wedge \cdots d\log z_\ell\wedge dz_{\ell+1}\wedge \cdots\wedge dz_n$ of $K_Y+D|_{\Omega}$, one has
   	$$
   	e^{-\varphi}:=|\sigma|^2_{h_{K_Y+D}}=\frac{i dz_1\wedge d\bar{z}_1\wedge \cdots \wedge i dz_n\wedge d\bar{z}_n }{|z_1\cdots z_\ell|^2\omega^n} \gtrsim\prod_{j=1}^{\ell}(-\log |z_i|^2)
   	$$ 
   	where the last inequality follows from \eqref{eq:bound}. Hence the   local potential  $\varphi$ of $h_{K_Y+D}=e^{-\varphi}$ is always locally bounded. One the other hand, $\hess \varphi=-\frac{1}{2\pi}{\rm Ric}(\omega)\geq 0$ over  the Zariski dense open set $U_0$. This implies that the curvature current   $\Theta_{h_{K_Y+D}}(K_Y+D)=\hess \varphi$ is positive everywhere. On the other hand, since $\varpi$ is the trivial extension of $\omega$, one has thus
   	\begin{align}\label{eq:compare}
   	\Theta_{h_{K_Y+D}}(K_Y+D)\geq c\varpi.
   	\end{align}
   	The lemma follows.
   \end{proof}

 \section{Pseudo-K\"ahler metrics induced by nilpotent harmonic bundles} 
 In this section we prove that 	the nilpotent harmonic bundle on a complex manifold $U$ induces a pseudo-K\"ahler metric with nice curvature properties similar to the case of period domains. 
\begin{proposition}\label{lem:kahler}
	Assume that \(U\) is a complex manifold that supports a  harmonic bundle $(E,\theta,h)$ so that $\theta:T_U\to \End(E)$ is injective at one point. Then \(U\) admits a pseudo-K\"ahler metric $\omega$ with non-positive holomorphic bisectional curvature. If $(E,\theta,h)$ is moreover nilpotent, then  the holomorphic sectional curvature of $\omega$ is bounded from above by $-\frac{1}{4^{\rank E-1}}$.
\end{proposition}
\begin{proof}
	We define a metric \(h_{U}\) as the pullback metric of \(h\) by the map \(T_{U} \stackrel{\theta}{\to} \End(E)\). This gives
	$$
	h_U(\xi_1,\xi_2):=\langle \theta(\xi_1),\theta(\xi_2)\rangle_h
	$$
	for any $\xi_1,\xi_2\in T_U$. The fundamental  $(1,1)$-form $\omega$ relative to $h_U$ can thus be written as
\begin{align}\label{eq:pseudo Kahler}
	\omega=-i {\rm tr}(\theta_h^*\wedge\theta), 
	\end{align}
	which shows that $\omega\geq 0$. Since $\theta:T_U\to \End(E)$ is immersive at one point, $\omega$ is therefore strictly positive at a general point. Moreover, 
	$$
	d\omega=-i d{\rm tr}(\theta_h^*\wedge\theta)=-i  {\rm tr}(D_h\theta_h^*\wedge\theta)+ i  {\rm tr}(\theta_h^*\wedge D_h\theta)
	$$
	where $D_h$ is the Chern connection of $(E,h)$. 
	Note that $D_h\theta=0=D_h\theta_h^*$, hence $d\omega=0$. Thus $\omega$ is a pseudo-K\"ahler form. 
	
Let $p\in U$ so that $T_U\to \End(E)$ is injective. Pick local coordinates $(z_1,\ldots,z_n)$ centered at $p$, and set $\theta_i:=\theta(\frac{\d}{\d z_i})$. Denote by $ {\theta}_i^*$ the adjoint of $\theta_i$ with respect to $h$. Write $R$ to be the curvature tensor of $\omega$, and denote by $R^{\End(E)}$ the curvature tensor of $\End(E)$ induced by the harmonic metric $h$. By  the curvature decreasing properties of subbundles, the holomorphic bisectional curvature in the direction $\frac{\d}{\d z_i}$ and $\frac{\d}{\d z_j}$ is
\begin{align*}
	R_{i\bar{j}j\bar{i}}:&=\langle  R_{i\bar{j}}(\frac{\d}{\d z_j}),\frac{\d}{\d z_i}\rangle_{h_U}\leq \langle  R^{\End(E)}_{i\bar{j}}(\theta_j),\theta_i\rangle_{h}.
\end{align*}

	By the flatness of \(D_{h} + \theta + \theta_{h}^{\ast}\), we have \(R^{E}_{i\overline{j}} = - [\theta_{i}, \theta_{j}^{\ast}]\), so \(R^{\End(E)}_{i\bar{j}}(\theta_k) = - [[\theta_{i}, \theta_{j}^{\ast}],\theta_{k}]\). This gives

\begin{align*} 
	R_{i\bar{j}j\bar{i}} & \leq -\langle  [[\theta_i,\theta_j^*], \theta_j ],\theta_i\rangle_{h} \\
& =-{\rm tr}([[\theta_i,\theta_j^*], \theta_j ] \theta^*_i)\\
&=-{\rm tr}( [\theta_i,\theta_j^*] [ \theta_j,  \theta^*_i])\\
&=-|[\theta_i,\theta_j^*] |^2\leq 0.
\end{align*}
We conclude that $\omega$ has non-positive holomorphic bisectional curvature at any point $p\in U$ where $\theta$ is injective. 

Assume now $(E,\theta,h)$ is moreover nilpotent. Then $\theta_i:E_x\to E_x$ is a nilpotent endomorphism for each $i$ and $x\in U$.   Recall that the holomorphic sectional curvature in the direction $\frac{\d}{\d z_i}$ is defined by
$$
K(\frac{\d}{\d z_i}):=\frac{R_{i\bar{i}i\bar{i}}}{|\frac{\d}{\d z_i}|^4}\leq \frac{-|[\theta_i,\theta_i^*] |^2}{|\theta_i|^4}
$$ 
Since $\theta_i$ is nilpotent,	by  \cref{lem:algebra} below, one has
$$
|   [\theta_i, {\theta}^*_i]|\geq \frac{|\theta_i|^2}{2^{\rank E-1}}.
$$
This proves that 
	$$K(\frac{\d}{\d z_i})\leq -\frac{1}{4^{\rank E-1}}.$$
Since the local coordinate is arbitrary, this proves that the holomorphic sectional curvature of $\omega$ is bounded from above by $-\frac{1}{4^{\rank E-1}}$.
	\end{proof}

The following lemma of linear algebra was outlined in \cite[p. 27]{Sim92}.
\begin{lem}\label{lem:algebra}
	Let $A$ be a nilpotent \(n \times n\)-matrix with values in the complex numbers, and let \(H\) be an hermitian definite positive matrix of size \(n\). Let \(A^{\ast} := H {}^{t} \overline{A} H^{-1}\) be the adjoint of \(A\) with respect to \(H\). Then \(|[A,A^*]|_{H}\geq \frac{1}{2^{n-1}}|A|_{H}^2\), where $|A|_{H}^{2} = \frac{1}{2^{n-1}}{\rm tr}(AA^*)$.
\end{lem}
\begin{proof}
	Since \(A\) is nilpotent, there is a strictly decreasing flag \(\mathbb{C}^{n} = F_{0} \supsetneq F_{1} \supsetneq \dotsc \supsetneq F_{m} = 0\) \((m \leq n)\) and such that \(A F_{i} \subset F_{i+1}\). Applying the standard orthonormalization process, we may assume that the flag \((F_{i})\) is adapted to a \(H\)-unitary base of \(\mathbb{C}^{n}\). Changing the standard base to this new base, we may now assume that \(A\) is strictly upper triangular, and \(H\) is the identity.
	\medskip
	
	Write $A:=(a_{ij})_{1\leq i,j\leq n}$. Denote by $[A,A^*]:=(b_{ij})_{1\leq i,j\leq n}$.  Since $A$ is strictly upper triangular, then 
	$$
	b_{ii}= \sum_{j=i+1}^{n}|a_{ij}|^2-\sum_{k=1}^{i-1}|a_{1k}|^2.
	$$
	Set $c_i:= \sum_{j=i+1}^{n}|a_{ij}|^2$. Then $\sum_{i=1}^{n-1}c_i=|A|^2$. There exists an integer $m$ with $1\leq m\leq n-1$ so that
$$ 
c_i<\frac{1}{2^{n-i}}|A|^2 $$
for $i<m$ and $$
c_m\geq\frac{1}{2^{n-m}}|A|^2. 
$$
Note that $b_{ii}\geq c_{i}-\sum_{j=1}^{i-1}c_j$. This implies $$b_{mm}\geq \frac{1}{2^{n-1}}|A|^2.$$
The lemma follows from the fact that $|[A,A^*]|^2\geq |b_{mm}|^2$. 
	\end{proof}
\section{Criterion for Picard hyperbolicity}
In this section we will establish our criterion for pseudo-Picard hyperbolicity of \emph{quasi-compact K\"ahler manifolds}. \cref{criteria1,criteria2} will be used to prove \cref{main,main2}  respectively.  Their proofs are inspired by the   Second Main theorem of Brotbek-Brunebarbe \cite{BB20} and by \cite[Lemma 5.1]{Yam19}.  Since we work on K\"ahler manifolds rather than projective ones,  we have to establish  the criterion on removable singularities of  holomorphic maps from  punctured disks into   {compact K\"ahler manifold} in term of the growth of Nevanlinna characteristic functions (see \cref{prop:criterion}).
\begin{thm}\label{thm:criterion}
Let $Y$ be a compact K\"ahler manifold, and let $D$ be a simple normal crossing divisor on $Y$. Assume that $U:=Y-D$ is equipped with a pseudo-K\"ahler metric $\omega$  whose holomorphic sectional curvature   is bounded from above by a negative constant $-2\pi c$, then 
\begin{thmlist}
	\item \label{criteria1}$U$ is  Picard hyperbolic modulo the non-K\"ahler locus $E_{nK}(\{\varpi\})$. Here $\varpi$ is the closed positive $(1,1)$-current on $Y$ which is the trivial extension of $\omega$, and its cohomology class $\{\varpi\}$ is big. Moreover, 
\begin{align}\label{eq:nK}
	E_{nK}(\{\varpi\})\subset Y- \{y\in U\mid \omega\ \mbox{is strictly positive at}\ y\}.
\end{align}
	\item \label{criteria2}If $c\{\varpi\}-\{D\}$ is a big class, then $Y$ is   Picard hyperbolic modulo    $E_{nK}(c\{\varpi\}-\{D\})\cup D$.
	\end{thmlist}  
\end{thm}
\begin{proof}
	Our first step will be to prove an inequality similar to the Arakelov-Nevanlinna inequality of \cite[Theorem~4.1]{BB20} (see \eqref{eq:first}). The method of using a current with Poincar\'e singularities to define a first Nevanlinna characteristic function is essentially the same; the arguments can be explained quite shortly in our context so we will recall them for completeness.
\medskip

For any   $f:\Delta^*\to Y$ with $f(\Delta^*)\not\subset D$, write $f^*\omega=i\sigma(z)dz\wedge d\bar{z}$.	Since $\omega$ has negative holomorphic sectional curvature,    $\sigma(z)\in L^1_{\rm loc}(\Delta^*)$, and 	
\begin{align}\label{eq:sec}
\hess \log |f'|_\omega^2\geq cf^*\omega 
\end{align}
outside $f^{-1}(D)$. Indeed, let $z_0\in \Delta^*$ be so that $f(z_0)\in D$.  By the Ahlfors-Schwarz lemma, around $z_0$ we have
$$
\sigma(z)\lesssim \frac{1}{|z-z_0|^2(\log |z-z_0|^2)^2}.
$$

	If \(\phi\) is a local potential for \(\omega\), this shows that \(\log \sigma(z)  - f^{\ast} \phi + \log | z - z_{0}|^{2}\) is locally bounded from above near \(z_{0}\), and thus extends as a psh function on the whole disk. Applying the \(dd^{c}\)-operator, one gets the inequality of \((1, 1)\)-currents
$$
\hess \log \sigma(z)\geq i\sigma(z)dz\wedge d\bar{z}-[f^{-1}D].
$$ 
Here $f^{-1}D$ is the reduced divisor on $\Delta^*$, and $[f^{-1}D]$ is the associated current. In other words, 
\begin{align}\label{eq:current}
		\hess \log |f'|_\omega^2\geq cf^*\omega-[f^{-1}(D)]. 
		\end{align}
		holds over the whole $\Delta^*$.

We now change our model of the disk into $\Delta^*:=\{z\in \bC\mid 1<|z|<\infty\}$ by taking $z\mapsto \frac{1}{z}$, and define a \emph{Nevanlinna characteristic function} 
$$
	T_{f,\omega}(r):=\int_{2}^{r}\frac{dt}{t}\int_{\Delta_{2,t}}f^*\omega$$  where $\Delta_{2,t}:=\{z\in \Delta^*\mid 2<|z|<t\}$. 
\medskip

By Jensen formula, one has
\begin{align}\label{eq:jensen}
	\int_{2}^{r}\frac{dt}{t}\int_{\Delta_{2,t}}	\hess \log |f'|_\omega^2 =&\int_{0}^{2\pi}\log |f'(re^{i\theta})|_{\omega}\frac{d\theta}{2\pi}-\int_{0}^{2\pi}\log |f'(2e^{i\theta})|_{\omega}\frac{d\theta}{2\pi}\\\nonumber&-2\log \frac{r}{2}\int_{0}^{2\pi}\frac{\d\log |f'|_\omega(2e^{i\theta})}{\d r}\frac{d\theta}{2\pi}.
\end{align}
Using concavity of log, we have
	$$ \int_{0}^{2\pi}\log |f'(re^{i\theta})|_{\omega}\frac{d\theta}{2\pi}\leq \frac{1}{2}\log\int_{0}^{2\pi}  |f'(re^{i\theta})|^2_{\omega}\frac{d\theta}{2\pi}.
$$
Note that
$$
\frac{1}{2\pi r}\frac{d}{d r}(r\frac{d}{dr}T_{f,\omega}(r))=\int_{0}^{2\pi}  |f'(re^{i\theta})|^2_{\omega}\frac{d\theta}{2\pi}.
$$ 
Since $T_{f,\omega}(r)$ and $r\frac{d}{dr}T_{f,\omega}(r)$ are both monotone increasing functions,  we apply  Borel's lemma \cite[Lemma 1.2.1]{NW14} twice so that, for any $\delta>0$ one has  
\begin{align*}
\log\Big(\frac{d}{d r}(r\frac{d}{dr}T_{f,\omega}(r))\Big)&\leq   (1+\delta)\log\Big(r\frac{d}{dr}T_{f,\omega}(r)\Big)\quad &\lVert\\
&=(1+\delta)\log r+(1+\delta)\log\Big( \frac{d}{d r}T_{f,\omega}(r)\Big) &\lVert\\
 &\leq (1+\delta)\log r+ (1+\delta)^2 \log T_{f,\omega}(r)\quad &\lVert
\end{align*} 
Here  $\lVert$  means that the inequality holds outside a Borel set $E\subset (2,\infty)$ of finite Lebesgue measure.  
The above inequalities yield
$$
\frac{1}{2}\log\int_{0}^{2\pi}  |f'(re^{i\theta})|^2_{\omega}\frac{d\theta}{2\pi}\leq \frac{(1+\delta)^2}{2}\log T_{f,\omega}(r)+\frac{\delta}{2}\log r-\frac{1}{2}\log (2\pi)  \quad \lVert.
$$
Putting this into \eqref{eq:jensen}, we get
\begin{align*}
	\int_{2}^{r}\frac{dt}{t}\int_{\Delta_{2,t}}	\hess \log |f'|_\omega^2  \leq \frac{(1+\delta)^2}{2}\log T_{f,\omega}(r)+\frac{\delta}{2}\log r-\frac{1}{2}\log (2\pi) \\-\int_{0}^{2\pi}\log |f'(2e^{i\theta})|_{\omega}\frac{d\theta}{2\pi}-2\log \frac{r}{2}\int_{0}^{2\pi}\frac{\d\log |f'|_\omega(2e^{i\theta})}{\d r}\frac{d\theta}{2\pi}. \quad \lVert. 
\end{align*}
By \eqref{eq:current}, this implies the requested inequality
\begin{align}\label{eq:first}
c_1\log T_{f,\omega}(r)+c_2\log r+c_3\geq cT_{f,\omega}(r)-N_{f,D}^{[1]}(r)\quad \lVert 
\end{align}
for some positive constants $c_1,c_2,c_3$. Here $N_{f,D}^{[1]}(r)$ is the truncated counting function defined by
$$
N_{f,D}^{[1]}(r):=\int_{2}^{r}\frac{dt}{t}\int_{\Delta_{2,t}} [f^{-1}(D)].
$$
Obviously, it is zero if $f$ avoids $D$. Note that in this case, we would have \(T_{f, \omega}(r) \leq \log(r)\), and  the requested extension of \(f\) follows directly  from \cref{prop:criterion}  below  if \(\omega\) were  the restriction of a K\"ahler form in $Y$ on $U$. Since \(\omega\) is merely {\em pseudo-K\"ahler} and is only defined on $U$, one needs some additional work, which makes use of the condition that $\omega$ is closed.
\medskip

By \cref{lem:crucial}, the trivial extension of $\omega$ over $Y$, denoted by $\varpi$, is a closed positive current.   Moreover, the   cohomology class $\{\varpi\}$ is big. By \cref{boucksom} there is a K\"ahler current $S_2\in \{\varpi\}$ which is a smooth K\"ahler form on $Y-E_{nK}(\{\varpi\})$. We also choose a smooth closed $(1,1)$-form $\eta\in \{\varpi\}$. 
\begin{claim}\label{claim1}
	Fix any smooth K\"ahler metric $\omega_Y$ over $Y$. For any $f:\Delta^*\to Y$ with $f(\Delta^*)\not\subset E_{nK}(\{\varpi\}) $, there are positive constants $c_i$ so that 
\begin{align}\label{eq:second}T_{f,\omega}(r)\geq T_{f,\eta}(r)-c_4\log T_{f,\omega_Y}(r)-c_4\log r-c_5 \quad \lVert \\\label{eq:third}
	T_{f,\eta}(r)\geq c_6T_{f,\omega_Y}(r)-c_7\log r-c_8\quad \lVert\\\label{eq:fourth}
	T_{f,\omega_Y}\geq c_9 T_{f,\omega}(r)-c_{10}\log r-c_{11}\quad \lVert
	\end{align}
\end{claim} 
\begin{proof}[Proof of \cref{claim1}]
	Write $D=\sum_{i=1}^{\ell}D_i$. Set $\sigma_i$ to be a section $H^0(Y,\cO_Y(D_i))$ defining $D_i$, and pick a  smooth metric  $h_i$  for  $\cO_Y(D_i)$.  
	Since   $\eta\in\{\varpi\}$, there is a quasi-psh function $\varphi\leq 0$ defined on $Y$ so that $\eta=\varpi-\hess \varphi$.	By \eqref{eq:loglog}, one has 
$$
\varphi\geq -\delta_1\log ( \prod_{i=1}^{\ell}\log ^2 |\ep\cdot \sigma_i|_{  h_i}^2) 
$$
for some $\delta_1>0$ and $\ep>0$. By Jensen's formula, one has
\begin{align*} 
T_{f,\omega}(r)-T_{f,\eta}(r)=\int_{0}^{2\pi}  \varphi\circ f(re^{i\theta}) \frac{d\theta}{2\pi}-\int_{0}^{2\pi}\varphi\circ f(2e^{i\theta})\frac{d\theta}{2\pi}- \log \frac{r}{2}\int_{0}^{2\pi}\frac{\d\varphi\circ f(2e^{i\theta})}{\d r}\frac{d\theta}{2\pi}\\
\geq -\delta_1\int_{0}^{2\pi}  \log (\prod_{i=1}^{\ell}\log^2  |\ep\cdot \sigma_i|_{ h_i}^2)\circ f(re^{i\theta}) \frac{d\theta}{2\pi}-c_4\log r-c_5
\end{align*}
By the concavity of log, one has
$$
-\int_{0}^{2\pi}  \log (\prod_{i=1}^{\ell}\log ^2 |\ep\cdot \sigma_i|_{\cdot h_i}^2)\circ f(re^{i\theta}) \frac{d\theta}{2\pi}\geq -2\sum_{i=1}^{\ell}\log\int_{0}^{2\pi}   (-\log |\ep\cdot \sigma_i|_{ h_i}^2)\circ f(re^{i\theta}) \frac{d\theta}{2\pi}
$$
Using Jensen formula again, one obtains
$$
\int_{0}^{2\pi}   (-\log |\ep\cdot \sigma_i|_{ h_i}^2)\circ f(re^{i\theta}) \frac{d\theta}{2\pi}\leq T_{f,\Theta_{h_i}(D_i)}(r)+O(\log r).
$$
\eqref{eq:second} follows from the fact that 
$$
T_{f,\Theta_{h_i}(D_i)}(r)\leq \delta_2T_{f,\omega_Y}(r) 
$$
for some positive constants $\delta_2$.

Since $S_2$ and $\eta$ are both in $\{\varpi\}$, there is a quasi-psh function $\phi\leq 0$ defined on $Y$ so that $\eta=S_2-\hess \phi$.  Since $S_2$ is smooth over $U_0:=Y-E_{nK}(\varpi)$, and $f(\Delta^*)\cap U_0\neq \varnothing$, $f^*S_2$ is thus well defined on $\Delta^*$. By Jensen formula again, so one has
\begin{align}\label{eq:sample}
T_{f,\eta}(r)-T_{f,S_2}(r) =&-\int_{0}^{2\pi}  \phi\circ f(re^{i\theta})  \frac{d\theta}{2\pi}+\int_{0}^{2\pi}\phi\circ f(2e^{i\theta})\frac{d\theta}{2\pi}-\log \frac{r}{2}\int_{0}^{2\pi}\frac{\d\phi\circ f(2e^{i\theta})}{\d r}\frac{d\theta}{2\pi}\\\nonumber
&\geq -c_7\log r-c_8.
\end{align}
On the other hand, $S_2\geq \ep \omega_Y$ for some constant $\ep>0$ since $S_2$ is a K\"ahler current, one has
$$
T_{f,S_2}(r)\geq \ep T_{f,\omega_Y}(r).
$$
This proves \eqref{eq:third}.

In a similar vein as in \eqref{eq:sample}, one can prove that
$$
T_{f,\eta}(r)-T_{f,\omega}(r)\geq -\delta_3 \log r-\delta_4.
$$
Since $T_{f,\omega_Y}(r)\geq \delta_5 T_{f,\eta}(r)$, \eqref{eq:fourth} follows.
	\end{proof}
Let us prove \cref{criteria1}. For any $f:\Delta^*\to U$ with $f(\Delta^*)\not\subset E_{nK}(\{\varpi\})$, one has $N_{f,D}^{[1]}(r)=0$. Putting \eqref{eq:second} and \eqref{eq:third} into \eqref{eq:first},  we immediately conclude that $T_{f,\omega_Y}(r)\sim  \log r$ when $r\to \infty$. This proves that $f$ extends across the point $\infty$ by \cref{prop:criterion} below, hence $U$ is Picard hyperbolic modulo $E_{nK}(\{\varpi\})$.

Let us prove \eqref{eq:nK}. By \cref{lem:crucial} $\{{\varpi}\}$ is big and nef. By \cref{eq:CT}, one has
\begin{align*} 
E_{nK}(\{\varpi\})={\rm Null}(\{\varpi\}):=\bigcup_{\int_{Z}\{\varpi\}^{\dim Z}=0}Z
\end{align*}
where  the union is taken over all positive dimensional irreducible analytic subvarieties $Z$ in $Y$. If $$Z\not\subset Y- \{y\in U\mid \omega\ \mbox{is strictly positive at}\ y\}.$$ by \eqref{eq:loglog} one has $$\int_{Z}\{\varpi\}^{\dim Z}=\int_{Z^{\rm reg}\cap U}\omega^{\dim Z}>0.$$ This yields \eqref{eq:nK} by \eqref{eq:CT}. \eqref{eq:nK} is proved.
\medskip

Let us now prove \cref{criteria2}. Since $c\{\varpi\}-\{D\}$ is big, by \cref{boucksom} one can take a K\"ahler current $S_3\in \{c\varpi\}-\{D\} $  which is smooth outside the non-K\"ahler locus   $E_{nK}(\{c\varpi\}-\{D\})$. Let $f:\Delta^*\to Y$ be a curve which is not contained in $E_{nK}(\{c\varpi\}-\{D\})\cup D$. Since   $\{S_3+[D]\}=\{c\varpi\}=c\{\eta\}$, similar arguments as \eqref{eq:sample} show that
$$
 cT_{f,\eta}(r)-T_{f,S_3+[D]}(r)\geq -c_9\log r-c_{10}.
$$ 
Moreover,
$$
T_{f,S_3+[D]}(r)=T_{f,S_3}(r)+T_{f,[D]}(r)\geq c_{11}T_{f,\omega_Y}(r)+N_{f,D}^{[1]}(r).
$$
Combining these inequalities with \eqref{eq:second}, \eqref{eq:first} and $\eqref{eq:fourth}$, we conclude that $T_{f,\omega_Y}(r)\sim \log r$. This proves that $f$ extends across the point $\infty$  by \cref{prop:criterion} below.
 	\end{proof} 
 We state and prove the following criterion on the  extendibility across the origin of the holomorphic map from the punctured disk to a \emph{compact K\"ahler manifold}.  
 \begin{lem}\label{prop:criterion}
 	Let $(Y,\omega_Y)$ be a compact K\"ahler manifold, and let $f:\Delta^*\to Y$ be a holomorphic map from the punctured disk to $Y$. If $$
 	T_{f,\omega_Y}(r):=\int_{2}^{r}\frac{dt}{t}\int_{\Delta_{2,t}}f^*\omega_Y$$  is bounded from above by $C\log r$ when $r\to \infty$ for some constant $C>0$. Here we consider our model of the punctured disk as $\Delta^*:=\{z\in \bC\mid 1<|z|<\infty\}$ by taking $z\mapsto \frac{1}{z}$.  Then $f$ extends to a holomorphic map $\Delta^*\cup\{\infty\}\to Y$.
 \end{lem}
\begin{proof}
	We  claim that    $\int_{\Delta_{2,t}}f^*\omega_Y<3C$ for any $t>0$.  Or else, there is $r_0>0$ so that $\int_{\Delta_{2,t}}f^*\omega_Y\geq 3C$  when $t\geq r_0$. Then 
	$$
 	T_{f,\omega_Y}(r)\geq 3C(\log r -\log r_0)\geq 2C\log r
	$$
	if $r\gg 0$.  This contradicts with our assumption.

For simplicity, let us now change our model of the punctured disk to $\Delta^*:=\{z\in \bC\mid 0<|z|<1\}$ by taking $z\mapsto \frac{1}{z}$. Then one has
	$$
\int_{\{z\in \bC\mid 0<|z|<\frac{1}{2}\}} f^*\omega_Y< 3C.
$$
	Consider the graph $V$ of $f$, which is an one dimensional closed analytic subvariety of  $\Delta^*\times Y$. Let us equip $\Delta\times Y$ with the K\"ahler metric $\omega':=q_1^*\omega_e+ q_2^*\omega_Y$, where $q_1:\Delta\times Y\to \Delta$ and $q_2:\Delta\times Y\to Y$ is the projection map, and $\omega_e:=idz\wedge d\bar{z}$. Then the volume of the analytic set    $V\cap  \{z\mid 0<|z|<\frac{1}{2}\}\times Y$ with respect to the K\"ahler metric $\omega'$ is 
	$$
	\int_{\{z\in \bC\mid 0<|z|<\frac{1}{2}\}} f^*\omega_Y+\omega_e\leq 3C+\pi.
	$$
 We now apply   \cref{thm:Bishop} to conclude that the closure of $V$ in $\Delta\times Y$, denoted by  $\overline{V}$, is an one dimensional closed analytic subset. Hence the map $q_1|_{\overline{V}}:\overline{V}\to \Delta$ is a \emph{proper} holomorphic map, which is an isomorphism over $\Delta^*$.  Therefore, $q_1|_{\overline{V}}$ is moreover an isomorphism. The composition $q_2\circ (q_1|_{\overline{V}})^{-1}:\Delta\to Y$ is a holomorphic map which extends $f$. The proposition is proved.
	\end{proof}
\begin{rem}
Note that \cref{prop:criterion} is a well-known result when $Y$ is a projective manifold; see \emph{e.g.} \cite[2.11. Cas \guillemotleft local \guillemotright]{Dem97b} or \cite[Lemma 6.5]{Siu15}.  For their strategy of the proof, they use sufficiently many global rational functions on $Y$ to reduce the theorem to holomorphic maps  $\Delta^* \to\bP^1$ and then apply Nevanlinna's logarithmic derivative lemma to conclude. Our proof of \cref{prop:criterion} thus  also provides an alternative and simplified proof in the projective setting. 
 \end{rem}

\section{Criterion for algebraic hyperbolicity}
In this section we will establish an algebraic analogue to \cref{thm:criterion}. 
\begin{thm}\label{thm:algebraic}
	Let $(Y,D)$ be a compact K\"ahler log pair. Assume that $U:=Y-D$ is equipped with a pseudo-K\"ahler metric $\omega$  whose holomorphic sectional curvature   is bounded above by a negative constant $-2\pi c$, then 
	\begin{thmlist}
		\item \label{acriteria1} $U$ is  algebraic hyperbolic modulo   $E_{nK}(\{\varpi\})$, where $\varpi$ is the closed positive $(1,1)$-current on $Y$ which is the trivial extension of $\omega$. 
		\item \label{acriteria2}If $c\{\varpi\}- \{D\}$ is a big class, then $Y$ is algebraically hyperbolic modulo    $E_{nK}(c\{\varpi\}- \{D\})\cup D$.
	\end{thmlist}  
\end{thm}
\begin{proof}
	By \cref{lem:crucial}, we know that $\{\varpi\}$ is big. 
Let $C$ be any  irreducible reduced curve not contained in $D\cup E_{nK}(\{\varpi\})$. Set $\nu:\tilde{C}\to C$ to be the normalization. Write $\tilde{C}^\circ:=\nu^{-1}(U)$, and denote by $P:=\nu^{-1}(D)$ the reduced divisor on $\tilde{C}$. By \eqref{eq:nK}, $\nu^*\omega$ is also a pseudo-K\"ahler metric on $\tilde{C}^\circ$. Since  the holomorphic sectional curvature of $\omega$  is bounded from above by a negative constant $-c$, by the curvature decreasing property, the holomorphic sectional curvature of $\nu^*\omega$ is also bounded above by $-2\pi c$. As in the proof of \cref{lem:pseudo}, $\nu^*\omega$ induces a singular hermitian metric $h_{K_{\tilde{C}}+P}$ whose curvature current is positive. Moreover, by \eqref{eq:compare}, one has
$$
 \Theta_{h_{K_{\tilde{C}}+P}}(K_{\tilde{C}}+P)\geq c\widetilde{\nu^*\omega}
$$
	where $\widetilde{\nu^*\omega}$ is the closed positive current on $\tilde{C}$ which is the trivial extension of $\nu^*\omega$. By \eqref{eq:loglog}, the Lelong numbers of the local potentials \(\phi\) of \(\varpi\) are \(0\), so using \(\nu^\ast \varpi \overset{\mathrm{loc}}{=} dd^{c}(\phi \circ \nu)\), one can easily check that $\widetilde{\nu^*\omega}=\nu^*\varpi$. 
Hence
\begin{align}\label{eq:elementary}
2g(\tilde{C})-2+i_Y(C,D)=\int_{\tilde{C}} \Theta_{h_{K_{\tilde{C}}+P}}(K_{\tilde{C}}+P)\geq c\int_{\tilde{C}}\nu^*\varpi=c\{C\}\cdot \{\varpi\},
\end{align}
where we use the notation in \cref{def:DC}. 

Fix a K\"ahler form $\omega_Y$ on $Y$. By \cref{boucksom} one can choose a K\"ahler current $S_1\in \{\varpi\}$ which is smooth outside $E_{nK}(\{\varpi\})$. Hence, there is a constant $\ep>0$ so that $S_1\geq \ep\omega_Y$. Since $C$ is not contained in $E_{nK}(\{\varpi\}$, one has
$$
\{C\}\cdot \{\varpi\}=\int_{\tilde{C}}\nu^*S_1\geq \ep\int_{\tilde{C}}\nu^*\omega_Y=\ep\deg_{\omega_Y}C.
$$
Putting this inequality into \eqref{eq:elementary}, we obtain 
$$
2g(\tilde{C})-2+i_Y(C,D)\geq c\ep\deg_{\omega_Y}C.
$$
The first claim follows since $c>0$ and $\ep>0$ does not depend on $C$.

If $c\{\varpi\}- \{D\}$ is big, by \cref{boucksom} again there is a K\"ahler current $S_2\in c\{\varpi\}- \{D\}$ which is smooth outside $E_{nK}(c\{\varpi\}- \{D\})$. Hence there is a constant $\ep_2>0$ so that $S_2\geq \ep_2\omega_Y$.  If $C$ is not contained in $D\cup E_{nK}(c\{\varpi\}- \{D\})$, by $ S_2+ [D]\in c\{\varpi\}$ one has
$$
c\{C\}\cdot \{\varpi\}=\int_{\tilde{C}}\nu^*(S_2+ D)\geq \ep_2\int_{\tilde{C}}\nu^*\omega_Y+ i_Y(C,D)=\ep_2\deg_{\omega_Y}C+ i_Y(C,D).
$$
Putting this to \eqref{eq:elementary}, we obtain 
$$
2g(\tilde{C})-2 \geq  \ep_2\deg_{\omega_Y}C.
$$
This proves the second claim.
	\end{proof}
\section{Proof of \cref{main}}
 We are now ready to prove \cref{main}.
\begin{proof}[Proof of \cref{main}]
	Take a compact K\"ahler manifold $Y$ compactifying $U$ so that $D:=Y-U$ is simple normal crossing. 	By \cref{lem:kahler}, the nilpotent harmonic bundle induces a pseudo-K\"ahler metric $\omega$  on $U$ whose holomorphic bisectional curvature is non-positive and holomorphic sectional curvature is bounded from above by $-\frac{1}{2^{\rank E-1}}$. One can  then apply the criterion in \cite[Theorem 2]{Cad16} or  \cite[Theorem 1.6]{BC17} to conclude that $U$ is of log general type. Alternatively, by \cref{lem:pseudo}, $K_Y+D\geq \frac{1}{4^{\rank E-1}\cdot 2 \pi}\{\varpi\}$ where $\varpi$ is the closed positive current on $Y$ which is the trivial extension of $\omega$. Since $\{\varpi\}$ is big, $K_Y+D$ is also big. This also proves that $U$ is of log general type.  Hence $Y$ is both a K\"ahler and Moishezon manifold, hence projective. 	By \cref{unique}, any compact complex manifold compactifying  $U$ is bimeromorphic to $Y$. This proves the uniqueness of algebraic structure of $U$ by Chow's theorem.
	
 It follows from \cref{criteria1,acriteria1} that $U$ is pseudo-Picard hyperbolic and pseudo-algebraically hyperbolic. 
\end{proof}

A direct consequence of \cref{criteria1,acriteria1} is the following result.
\begin{cor}
	Let $U$ be a quasi-projective manifold. If $U$ is equipped with a K\"ahler metric $\omega$ with holomorphic sectional curvature bounded from above by a negative constant, then $U$ is Picard hyperbolic and algebraically hyperbolic. \qed
\end{cor}

The above result gives a new proof of the following theorem by Borel \cite{Bor72}, Kobayashi-Ochiai \cite{KO71} and Pacienza-Rousseau \cite{PR07}.
\begin{thm}
	Let $U$ be a quasi-projective quotient of bounded symmetric domain by a torsion free lattice. Then $U$ is Picard hyperbolic and algebraically hyperbolic. 
\end{thm}
\begin{proof}
	Since the Bergman metric on $U$ is K\"ahler  with holomorphic sectional curvature bounded from above by a negative constant, the Picard hyperbolicity and algebraic hyperbolicity of $U$ follows from the above corollary immediately.
\end{proof}


  \section{Hyperbolicity for the compactification after finite unramified cover}\label{sec:strong}  
  In this section we will prove \cref{main2} using ideas similar to \cite[Proof of Theorem 5.1]{Den21}.
 \begin{thm}\label{strong}
 	Let $(Y,D)$ be a compact K\"ahler log pair.  Assume that    there is a nilpotent harmonic bundle $(E,\theta,h)$ on $U:=Y-D$ so that $\theta:T_{U}\to \End(E)$ is injective at one point.   Then there is a log morphism $\mu:(X,\tilde{D})\to (Y,D)$  from a projective log pair $(X,\tilde{D})$ which is a finite unramified cover over $U$ such that
 	\begin{thmlist}
 		\item \label{general type}  any irreducible subvariety  of $X$ non contained in  the analytic subvariety  $\mu^{-1}(E_{nK}(\{\varpi\}))\cup \tilde{D}$  is  of  general type; 
 		\item  \label{pseudo-Picard}   $X$  is    Picard hyperbolic modulo $\mu^{-1}(E_{nK}(\{\varpi\}))\cup \tilde{D}$;
 		\item \label{algebraically hyperbolic} $X$ is algebraically hyperbolic modulo $\mu^{-1}(E_{nK}(\{\varpi\}))\cup \tilde{D}$.
 	\end{thmlist}
 Here  $\varpi$ is the trivial extension of the pseudo-K\"ahler form $\omega=-i{\rm tr}(\theta_h^*\wedge\theta)$ on $U$ defined in \eqref{eq:pseudo Kahler}. Moreover, we have
\begin{align}\label{eq:inclusion}
 E_{nK}(\{\varpi\}) \subset   Y- \{y\in U\mid \theta\ \mbox{is injective at}\ y\}.
 \end{align}
 \end{thm}  
	We will need the following  crucial result proved in \cite[Claim 5.2]{Den21} to find the desired covering $\mu: {X}\to Y$ in \cref{strong}. The proof is based on residual  finiteness of the global monodromy group and Cauchy's argument theorem.
 \begin{lem}\label{claim}
Let $Y$ be a projective manifold and let $D=\sum_{j=1}^{\ell}D_j$ be a simple normal crossing divisor on $Y$. Assume that there is a complex local system $\mathcal{L}$ over $U:=Y-D$. Then for any $m>0$,  there is a smooth projective log pair $(X,\tilde{D})$ and a log morphism $\mu:(X,\tilde{D}=\sum_{i=1}^{N}\tilde{D}_i)\to (Y,D)$ which is unramified over $U$ so that for each $j=1,\ldots,\ell$,   one has
 	\begin{itemize}
 		\item either  ${\rm ord}_{\tilde{D}_j}(\mu^*D) \geq m,$ 
 		\item or the local monodromy group of $\mu^*\mathcal{{L}}$ around $\tilde{D}_j$ is trivial.\qed
 	\end{itemize} 
 \end{lem}

Let us now prove \cref{strong}.
\begin{proof}[Proof of \cref{strong}] By the proof of \cref{main},   $Y$ is a projective manifold.    For the $(1,1)$-form  $\omega$ on $U$ defined by
	\begin{align}\label{eq:pseudo Kahler2}
	\omega=-i {\rm tr}(\theta_h^*\wedge\theta),
	\end{align}  by \cref{lem:pseudo}, we know that $\omega$ is a pseudo-K\"ahler form whose holomorphic sectional curvature is bounded from above by $-\frac{1}{4^{\rank E-1}}$. Let $\varpi$ be the positive closed $(1,1)$-current on $Y$ which is the trivial extension of $\omega$. By \cref{lem:crucial}, the class $\{\varpi\}$ is big and nef. Choose $\ep>0$ so that $\{\varpi\}-\ep D$ is still big and 
\begin{align}\label{eq:equal}
E_{nK}(\{\varpi\}-\ep \{D\})=E_{nK}(\{\varpi\}).
\end{align} 
Pick $m\gg 0$ so that $m\ep\geq 2^{2\rank E-1} \pi$. Let $\mathcal{{L}}$ be the local system relative to the tame harmonic bundle. By \cref{claim}, we find a log morphism $\mu:(X,\tilde{D}=\sum_{i=1}^{N}\tilde{D}_i)\to (Y,D)$ from a smooth projective log pair $(X,\tilde{D})$ which is unramified over $U$ satisfying the properties therein.  
	
Set $D_X\subset \tilde{D}$ to be the sum of  all $\tilde{D}_j$'s so that the local monodromy group of $\mu^*\mathcal{L}$ around $\tilde{D}_j$ is \emph{not} trivial. Then by the dichotomy in \cref{claim},  $
\mu^*D-mD_X
$ is an effective divisor, and the monodromy of $\mu^*\mathcal{L}$  around $\tilde{D}_i$ with $\tilde{D}_i\not\subset D_X$ is trivial. By \cref{prop:extension} below, the   pull-back harmonic bundle extends to a nilpotent harmonic bundle  over $X-D_X$. Such a nilpotent harmonic bundle induces a pseudo-K\"ahler metric $\omega_2$ on $X-D_X$. One has $\omega_2=\mu^*\omega$ over $\tilde{U}$. $\omega_2$ thus has non-positive holomorphic bisectional curvature and holomorphic sectional curvature bounded from above by $-\frac{1}{4^{\rank E-1}}$. Denote by $\varpi_2$ the closed positive current which is the trivial extension of $\omega_2$. 
\begin{claim}\label{claim3}
$\mu^*\varpi=\varpi_2$.
\end{claim} 
\begin{proof}[Proof of \cref{claim3}]
By the very definition of trivial extension, we have
\begin{equation} \label{eq:bound2}
\mu^*\varpi \geq \mu^*\varpi-\varpi_2\geq 0.
\end{equation}
	Pick any point $x\in \tilde{D}$, and choose admissible coordinates $(\Omega;x_1,\ldots,x_n)$ and $(\Omega_2;y_1,\ldots,y_n)$ around $x$ and $y=\mu(x)$ with $\mu(\Omega_1)\subset \Omega_2$ so that $\Omega_1\cap \tilde{D}=(x_1\cdots x_{\ell_1}=0)$ and $\Omega_2\cap D=(y_1\cdots y_{\ell_2}=0)$. Since $\mu$ is a log morphism, one has $\mu^*y_i=g_i(x)\prod_{j=1}^{\ell_1}x_{j}^{a_{ij}}$ with $a_{ij}\in \bZ_{\geq 0}$ and $g_i(x)\in \cO(\Omega_1)$. By \eqref{eq:loglog}, the local potential $\phi$ of $\varpi=\hess \phi$ satisfies
	\begin{align*} 
	\phi\gtrsim -\log \Big(\prod_{j=1}^{\ell_1}(-\log |y_j|^2)\Big).
	\end{align*}
Hence the local potential $ \phi\circ \mu$ of $\mu^*\varpi=\hess \phi\circ \mu$ satisfies
	\begin{align*} 
\phi\circ \mu\gtrsim -\log \Big(\prod_{j=1}^{\ell_2}(-\log |x_j|^2)\Big).
\end{align*}
	Therefore, the Lelong numbers of $\mu^*\varpi$ are zero everywhere, and by \eqref{eq:bound2}, the same holds for the positive current $\mu^*\varpi-\varpi_2$. On the other hand, since $\omega_2=\mu^*\omega$,  $\mu^*\varpi-\varpi_2$  is thus supported on $\tilde{D}$. By the support theorem \cite[(2.14) Corollary]{Dembook}, $\mu^*\varpi-\varpi_2=\sum_{i=1}^{N}\lambda_i[\tilde{D}_i]$ with $\lambda_i\geq 0$. Hence $\mu^*\varpi-\varpi_2=0$.
	\end{proof}
Note that
\begin{align*}
\mu^*(\{\varpi\}-\ep \{D\})&=\{\varpi_2\}-\ep \{\mu^*D-mD_X\}-m\ep \{D_X\}\\&= \{\varpi_2\}-2^{2\rank E-1} \pi\{D_X\}-\ep \{\mu^*D-mD_X\}-(m\ep-2^{2\rank E-1} \pi)  \{D_X\}. 
\end{align*}
Recall that $\mu^*D-mD_X\geq 0$, and $m\ep\geq 2^{2\rank E-1} \pi$. Hence 
$$ \{\varpi_2\}-2^{2\rank E-1} \pi\{D_X\} =\mu^*(\{\varpi\}-\ep \{D\})+\{D'\}$$
where $D'$ is some  effective $\bR$-divisor supported in $\tilde{D}$. Therefore, $\{\varpi_2\}-2^{2\rank E-1} \pi\{D_X\}$ is big with its non-K\"ahler locus
$$
E_{nK}(\{\varpi_2\}-2^{2\rank E-1} \pi\{D_X\})\subset E_{nK}(\mu^*(\{\varpi\}-\ep \{D\}))\cup \tilde{D}.
$$
Applying \cref{lem:nK} below to $ \{\varpi\}-\ep \{D\}$,   we obtain
$$
 E_{nK}(\mu^*(\{\varpi\}-\ep \{D\}))\subset  \mu^{-1}(E_{nK}(\{\varpi\}-\ep \{D\}))\cup \tilde{D}.
$$
By \eqref{eq:equal}, one has
\begin{align}\label{eq:final}
E_{nK}(\{\varpi_2\}-2^{2\rank E-1} \pi\{D_X\})\subset \mu^{-1}(E_{nK}(\{\varpi\}))\cup \tilde{D}. 
\end{align}
Recall that the holomorphic sectional curvature of $\omega_2$ is bounded from above by $-\frac{1}{4^{\rank E-1}}$. By \cref{criteria2,acriteria2}, we conclude that $X$ is both Picard hyperbolic and  algebraically hyperbolic modulo $\mu^{-1}(E_{nK}(\{\varpi\}))\cup \tilde{D}$. \cref{pseudo-Picard,algebraically hyperbolic} follows. 

Let $\tilde{Z}\subset X$ be any irreducible closed subvariety which is not contained in $\mu^{-1}(E_{nK}(\{\varpi\}))\cup \tilde{D}$.  Let $g:Z\to \tilde{Z}$ be a desingularization so that $D_Z:=g^{-1}(D_X)$ is a simple normal crossing divisor. Applying \cref{boucksom} we can pick a K\"ahler current 
$$
S\in \{\varpi_2\}-2^{2\rank E-1} \pi\{D_X\}
$$
so that $S$ is smooth outside $\mu^{-1}(E_{nK}(\{\varpi\}))\cup \tilde{D}$ by \eqref{eq:final}. Since $g(\tilde{Z})$ is not contained in $\mu^{-1}(E_{nK}(\{\varpi\}))\cup \tilde{D}$, the pull-back $g^*S$ exists and is a closed positive $(1,1)$ current in $g^*\{\varpi_2\}-2^{2\rank E-1} \pi g^*\{D_X\}$. Hence $g^*\{\varpi_2\}-2^{2\rank E-1} \pi g^*\{D_X\}$ is pseudo effective.

Write $Z^\circ:=Z-D_Z$. We claim that $\omega_3:=g^*\omega_2$ is strictly positive at one point of $Z^\circ$, hence is a pseudo K\"ahler form on $Z^\circ$. Or else, 
$$\int_{\tilde{Z}}\{\varpi_2\}^{\dim \tilde{Z}}=\int_{Z^\circ }(g^*\omega_2)^{\dim Z}=0,$$
which implies that $\tilde{Z}\in E_{nK}(\{\varpi_2\})$ by \cref{eq:CT}.  Since $$E_{nK}(\{\varpi_2\})\subset E_{nK}(\{\varpi_2\}-2^{2\rank E-1} \pi\{D_X\})\subset \mu^{-1}(E_{nK}(\{\varpi\}))\cup \tilde{D},$$
this contradicts with the assumption that  $\tilde{Z}\subset X$  is not contained in $\mu^{-1}(E_{nK}(\{\varpi\}))\cup \tilde{D}$. Therefore, $\omega_3$ is a pseudo K\"ahler form.

By the curvature decreasing property of submanifolds, we conclude that the holomorphic bisectional curvature of $\omega_3$ is non-positive, and the holomorphic sectional curvature of $\omega_3$ is bounded from above by $-\frac{1}{4^{\rank E-1}}$. Let $\varpi_3$ be the closed positive $(1,1)$ current  on $Z$ which is the trivial extension of $\omega_3$.  One can employ a similar method as for \cref{claim3} to show that
$$
g^*\{\varpi_2\}=\{\varpi_3\}.
$$
Recall that   $g^*\{\varpi_2\}-2^{2\rank E-1} \pi g^*\{D_X\}$ is pseudo effective. Since $g^*D_X\geq D_Z$, $ \{\varpi_3\}-2^{2\rank E-1} \pi  \{D_Z\}$ is thus also  pseudo effective. By \cref{lem:pseudo}, one has
$$
\frac{1}{2^{2\rank E-1} \pi}\{\varpi_3\}\leq K_Z+D_Z.
$$  
Hence $K_Z$ is big. \cref{general type} follows.

Lastly, by \eqref{eq:pseudo Kahler2}, $\omega$ is strictly positive at any point where $\theta$ is injective.   \eqref{eq:inclusion} then follows from \eqref{eq:nK}.
\end{proof} 

We state and prove the following crucial extension result for nilpotent tame harmonic bundles across the boundary components around which the local monodromies of the corresponding local system are trivial. Its proof was communicated to us by C. Simpson, and it uses the deep theorem by Mochizuki  on the correspondence between tame pure imaginary harmonic bundles and semisimple local systems over quasi-projective manifolds.

\begin{proposition}\label{prop:extension}
	Let $Y$ be a projective manifold and let $D=\sum_{i=1}^{m}D_i$ be a simple normal crossing divisor on $Y$. Let  $(E,\theta,h)$ be a nilpotent harmonic bundle on $U:=Y-D$, whose corresponding complex local system is denoted by $\mathcal{L}$. Assume that for $i=1,\ldots,r$ the local monodromy of $\mathcal{L}$ around the component  $D_i$ is trivial. Then $(E,\theta,h)$ extends to a nilpotent harmonic bundle on $U':=Y-\sum_{i=r+1}^{m}D_i$.
\end{proposition}
\begin{proof}
 Since $\theta$ is assumed to be nilpotent, the eigenvalue  of the residue  ${\rm Res}(\theta)$ at each component  $D_i$ is thus zero. Hence $(E,\theta,h)$ is a \emph{tame pure imaginary harmonic bundle} in the sense of \cite[Definition 22.3]{Moc07b}. By \cite[Proposition 22.15]{Moc07b}, $\mathcal{L}$ is semisimple. Hence it is a direct sum $\mathcal{L}:=\oplus_\alpha \mathcal{L}_\alpha\otimes \bC^{m_\alpha}$, where $\mathcal{L}_\alpha$  is a  simple local system and $m_\alpha>0$. Since the local monodromy of $\mathcal{L}$ around the component $D_i$ is trivial for $i=1,\ldots,r$, so is $\cL_\alpha$ for each $\alpha$. Hence $\mathcal{L}_\alpha$ extends to a local system $\mathcal{L}_\alpha'$ on $U'$.  Since the map between   fundamental groups $\pi_1(U)\to \pi_1(U')$ is surjective, $\mathcal{L}_\alpha'$ is thus also  simple.

 By \cite[Theorem 25.21]{Moc07b}, there is a tame pure imaginary harmonic bundle $(E_\alpha,\theta_\alpha,h_\alpha)$ on $U$ whose corresponding local system is $\cL_\alpha$. Moreover, by the uniqueness property of the correspondence between semisimple local systems and tame pure imaginary harmonic bundles proved in \cite[Theorem 25.28]{Moc07b}, one has
 $$
 (E,\theta,h)=\oplus_\alpha\big(\oplus_{m_\alpha}(E_\alpha,\theta_\alpha) , h_\alpha\otimes g_\alpha\big),
 $$
 where $g_\alpha$ denotes a hermitian metric of $\bC^{m_\alpha}$. Since $t^{\rank E}=\det (t-\theta)=\prod_{\alpha}\det (t-\theta_\alpha)^{m_\alpha}$, one has $\det (t-\theta_\alpha)=t^{\rank E_\alpha}$. Hence each $(E_\alpha,\theta_\alpha,h_\alpha)$ is a nilpotent harmonic bundle.

  Again by \cite[Theorem 25.21]{Moc07b}, there is a tame pure imaginary harmonic bundle on $(E'_\alpha,\theta_\alpha',h_\alpha')$ on $U'$ whose corresponding local system is $\mathcal{L}_\alpha'$. The restriction   $(E'_\alpha,\theta_\alpha',h_\alpha')|_{U}$  is thus a tame pure imaginary harmonic bundle  with the corresponding local system $\cL_\alpha'|_{U}=\mathcal{L}_\alpha$. By the uniqueness result in \cite[Theorem 25.28]{Moc07b}, $(E_\alpha',\theta_\alpha')|_{U}=(E_\alpha,\theta_\alpha)$ and $c_\alpha\cdot  h'_\alpha|_{U}= h_\alpha$ for some constant $c_\alpha>0$.  Since characteristic polynomial $\det(t-\theta_\alpha')|_{U}=\det(t-\theta_\alpha)=t^{\rank E_\alpha}$, by continuity $\det(t-\theta_\alpha')=t^{\rank E_\alpha}$ on $U'$. Hence $(E'_\alpha,\theta_\alpha',h_\alpha')$ is also nilpotent.   Therefore, the nilpotent harmonic bundle 
    $$
   (E',\theta',h')=\oplus_\alpha\big(\oplus_{m_\alpha}(E'_\alpha,\theta'_\alpha), c_\alpha\cdot h'_\alpha\otimes  g_\alpha\big) 
   $$
   defined on $U'$ extends $(E,\theta,h)$. The proposition is proved.
	\end{proof} 
 
The following result allows us to control non-K\"ahler locus on ramified covers.
 \begin{lem}\label{lem:nK}
 	Let $\mu:(X,\tilde{D})\to (Y,D)$ be a log morphism between  compact K\"ahler log pairs, which is unramified over $X-\tilde{D}$.  Let $\alpha$ be a big class on $Y$. Then
 	\begin{align}
 	E_{nK}(\mu^*\alpha)\subset \mu^{-1}(E_{nK}(\alpha))\cup \tilde{D}
 	\end{align}
 \end{lem}
\begin{proof}
By \cref{boucksom}, one can take a K\"ahler current $T\in \alpha$ with analytic singularities which is smooth outside $E_{nK}(\alpha)$. 	Choose a K\"ahler form $\omega$ on $Y$ so that $T\geq \omega$. Then $\{\mu^*\omega\}$ is a big and nef class, and by \cref{eq:CT}, one has
$$
E_{nK}(\{\mu^*\omega\})\subset \tilde{D}.
$$
Applying \cref{boucksom} again, there is a global quasi-psh function $\varphi$ on $X$ with analytic singularities which is smooth outside $\tilde{D}$ so that
 $
\mu^*\omega+\hess \varphi
$ 
is a K\"ahler current on $X$. It follows from $T\geq \omega$ that
 $
\mu^*T+\hess \varphi\geq \mu^*\omega+\hess \varphi
$. Hence $\mu^*T+\hess \varphi$ is a K\"ahler current with analytic singularities, which is smooth outside $\mu^{-1}(E_{nK}(\alpha))\cup \tilde{D}$. Since $\mu^*T+\hess \varphi\in \mu^*\alpha$, by the very definition of non-K\"ahler locus \cref{def:nK}, one has
$$
E_{nK}(\mu^*\alpha)\subset \mu^{-1}(E_{nK}(\alpha))\cup \tilde{D}.
$$
The lemma is proved.
	\end{proof}


\begin{thebibliography}{{Den}20b}
	
	\bibitem[BB20]{BB20}
	\textsc{D.~{Brotbek} and Y.~{Brunebarbe}}.
	\newblock {Arakelov-Nevanlinna inequalities for variations of Hodge structures
		and applications}.
	\newblock \emph{arXiv e-prints}, (2020) arXiv:2007.12957.
	
	\bibitem[BBT18]{BBT18}
	\textsc{B.~{Bakker}, Y.~{Brunebarbe} and J.~{Tsimerman}}.
	\newblock {o-minimal GAGA and a conjecture of Griffiths}.
	\newblock \emph{arXiv e-prints}, (2018) arXiv:1811.12230.
	
	\bibitem[BC20]{BC17}
	\textsc{Y.~{Brunebarbe} and B.~{Cadorel}}.
	\newblock {Hyperbolicity of varieties supporting a variation of Hodge
		structure}.
	\newblock \emph{{Int. Math. Res. Not.}}, \textbf{2020}(2020) 1601--1609.
	
	\bibitem[{Bis}64]{Bis64}
	\textsc{E.~{Bishop}}.
	\newblock {Conditions for the analyticity of certain sets}.
	\newblock \emph{{Mich. Math. J.}}, \textbf{11}(1964) 289--304.
	
	\bibitem[Bor72]{Bor72}
	\textsc{A.~Borel}.
	\newblock Some metric properties of arithmetic quotients of symmetric spaces
	and an extension theorem.
	\newblock \emph{J. Differential Geometry}, \textbf{6}(1972) 543--560.
	
	\bibitem[{Bou}02]{Bou02}
	\textsc{S.~{Boucksom}}.
	\newblock {On the volume of a line bundle}.
	\newblock \emph{{Int. J. Math.}}, \textbf{13}(2002) 1043--1063.
	
	\bibitem[Bou04]{Bou04}
	\textsc{S.~Boucksom}.
	\newblock Divisorial {Z}ariski decompositions on compact complex manifolds.
	\newblock \emph{Ann. Sci. \'{E}cole Norm. Sup. (4)}, \textbf{37}(2004) 45--76.
	
	\bibitem[{Bru}18]{Bru16}
	\textsc{Y.~{Brunebarbe}}.
	\newblock Symmetric differentials and variations of {H}odge structures.
	\newblock \emph{J. Reine Angew. Math.}, \textbf{743}(2018) 133--161.
	
	\bibitem[{Bru}20a]{Bru16a}
	\textsc{Y.~{Brunebarbe}}.
	\newblock {A strong hyperbolicity property of locally symmetric varieties}.
	\newblock \emph{{Ann. Sci. \'Ec. Norm. Sup\'er. (4)}}, \textbf{53}(2020)
	1545--1560.
	
	\bibitem[{Bru}20b]{Bru20}
	\textsc{Y.~{Brunebarbe}}.
	\newblock {Increasing hyperbolicity of varieties supporting a variation of
		Hodge structures with level structures}.
	\newblock \emph{arXiv e-prints}, (2020) arXiv:2007.12965.
	
	\bibitem[{Cad}16]{Cad16}
	\textsc{B.~{Cadorel}}.
	\newblock {Symmetric differentials on complex hyperbolic manifolds with cusps}.
	\newblock \emph{arXiv e-prints}, (2016) arXiv:1606.05470. To appear in \emph{J.
		Differential Geom.}
	
	\bibitem[{Cad}18]{Cad18}
	\textsc{B.~{Cadorel}}.
	\newblock {Subvarieties of quotients of bounded symmetric domains}.
	\newblock \emph{arXiv e-prints}, (2018) arXiv:1809.10978.
	
	\bibitem[Che04]{Che04}
	\textsc{X.~Chen}.
	\newblock On algebraic hyperbolicity of log varieties.
	\newblock \emph{Commun. Contemp. Math.}, \textbf{6}(2004) 513--559.
	
	\bibitem[CT15]{CT15}
	\textsc{T.~C. {Collins} and V.~{Tosatti}}.
	\newblock {K\"ahler currents and null loci}.
	\newblock \emph{{Invent. Math.}}, \textbf{202}(2015) 1167--1198.
	
	\bibitem[Dem92]{Dem92}
	\textsc{J.-P. Demailly}.
	\newblock Regularization of closed positive currents and intersection theory.
	\newblock \emph{J. Algebraic Geom.}, \textbf{1}(1992) 361--409.
	
	\bibitem[Dem97a]{Dem97}
	\textsc{J.-P. Demailly}.
	\newblock Algebraic criteria for {K}obayashi hyperbolic projective varieties
	and jet differentials.
	\newblock In \emph{Algebraic geometry---{S}anta {C}ruz 1995}, vol.~62 of
	\emph{Proc. Sympos. Pure Math.}, 285--360. Amer. Math. Soc., Providence, RI
	(1997).
	
	\bibitem[Dem97b]{Dem97b}
	\textsc{J.-P. Demailly}.
	\newblock Vari\'et\'es hyperboliques et \'equations diff\'erentielles
	alg\'ebriques.
	\newblock \emph{Gaz. Math.}, (1997) 3--23.
	
	\bibitem[Dem12a]{Dem12}
	\textsc{J.-P. Demailly}.
	\newblock \emph{Analytic methods in algebraic geometry}, vol.~1 of
	\emph{Surveys of Modern Mathematics}.
	\newblock International Press, Somerville, MA; Higher Education Press, Beijing
	(2012).
	
	\bibitem[Dem12b]{Dembook}
	\textsc{J.-P. Demailly}.
	\newblock \emph{Complex analytic and differential geometry}.
	\newblock Online e-book (2012).
	
	\bibitem[{Den}20a]{Den20}
	\textsc{Y.~{Deng}}.
	\newblock {A characterization of complex quasi-projective manifolds uniformized
		by unit balls}.
	\newblock \emph{arXiv e-prints}, (2020) arXiv:2006.16178.
	
	\bibitem[{Den}20b]{Den21}
	\textsc{Y.~{Deng}}.
	\newblock {Big Picard theorem and algebraic hyperbolicity for varieties
		admitting a variation of Hodge structures}.
	\newblock \emph{arXiv e-prints}, (2020) arXiv:2001.04426.
	
	\bibitem[DLSZ19]{DLSZ}
	\textsc{Y.~{Deng}, S.~{Lu}, R.~{Sun} and K.~{Zuo}}.
	\newblock {Picard theorems for moduli spaces of polarized varieties}.
	\newblock \emph{arXiv e-prints}, (2019) arXiv:1911.02973.
	
	\bibitem[{Ete}20]{Ete20}
	\textsc{A.~{Etesse}}.
	\newblock {Complex-analytic intermediate hyperbolicity, and finiteness
		properties}.
	\newblock \emph{arXiv e-prints}, (2020) arXiv:2011.12583.
	
	\bibitem[HR21]{HR21}
	\textsc{Y.~{He} and M.~{Ru}}.
	\newblock {Nevanlinna Pair and Algebraic Hyperbolicity}.
	\newblock \emph{arXiv e-prints}, (2021) arXiv:2102.04624.
	
	\bibitem[JK18]{JK18}
	\textsc{A.~{Javanpeykar} and R.~A. {Kucharczyk}}.
	\newblock {Algebraicity of analytic maps to a hyperbolic variety}.
	\newblock \emph{arXiv e-prints}, (2018) arXiv:1806.09338.
	
	\bibitem[KO71]{KO71}
	\textsc{S.~Kobayashi and T.~Ochiai}.
	\newblock Satake compactification and the great {P}icard theorem.
	\newblock \emph{J. Math. Soc. Japan}, \textbf{23}(1971) 340--350.
	
	\bibitem[{Mil}13]{Mil13}
	\textsc{J.~S. {Milne}}.
	\newblock {Shimura varieties and moduli}.
	\newblock In \emph{{Handbook of moduli. Volume II}}, 467--548. Somerville, MA:
	International Press; Beijing: Higher Education Press (2013).
	
	\bibitem[{Moc}07]{Moc07b}
	\textsc{T.~{Mochizuki}}.
	\newblock \emph{{Asymptotic behaviour of tame harmonic bundles and an
			application to pure twistor $D$-modules. II}}, vol. 870.
	\newblock Providence, RI: American Mathematical Society (AMS) (2007).
	
	\bibitem[{Mum}77]{Mum77}
	\textsc{D.~{Mumford}}.
	\newblock {Hirzebruch's proportionality theorem in the non-compact case.}
	\newblock \emph{{Invent. Math.}}, \textbf{42}(1977) 239--272.
	
	\bibitem[{Nad}89]{Nad89}
	\textsc{A.~M. {Nadel}}.
	\newblock {The nonexistence of certain level structures on abelian varieties
		over complex function fields.}
	\newblock \emph{{Ann. Math. (2)}}, \textbf{129}(1989) 161--178.
	
	\bibitem[NW14]{NW14}
	\textsc{J.~Noguchi and J.~Winkelmann}.
	\newblock \emph{Nevanlinna theory in several complex variables and
		{D}iophantine approximation}, vol. 350 of \emph{Grundlehren der
		Mathematischen Wissenschaften [Fundamental Principles of Mathematical
		Sciences]}.
	\newblock Springer, Tokyo (2014).
	
	\bibitem[PR07]{PR07}
	\textsc{G.~Pacienza and E.~Rousseau}.
	\newblock On the logarithmic {K}obayashi conjecture.
	\newblock \emph{J. Reine Angew. Math.}, \textbf{611}(2007) 221--235.
	
	\bibitem[Rou16]{Rou16}
	\textsc{E.~Rousseau}.
	\newblock Hyperbolicity, automorphic forms and {S}iegel modular varieties.
	\newblock \emph{Ann. Sci. \'{E}c. Norm. Sup\'{e}r. (4)}, \textbf{49}(2016)
	249--255.
	
	\bibitem[{Sim}92]{Sim92}
	\textsc{C.~T. {Simpson}}.
	\newblock {Higgs bundles and local systems}.
	\newblock \emph{{Publ. Math., Inst. Hautes \'Etud. Sci.}}, \textbf{75}(1992)
	5--95.
	
	\bibitem[{Siu}75]{Siu75}
	\textsc{Y.-T. {Siu}}.
	\newblock {Extension of meromorphic maps into K\"ahler manifolds.}
	\newblock \emph{{Ann. Math. (2)}}, \textbf{102}(1975) 421--462.
	
	\bibitem[Siu15]{Siu15}
	\textsc{Y.-T. Siu}.
	\newblock Hyperbolicity of generic high-degree hypersurfaces in complex
	projective space.
	\newblock \emph{Invent. Math.}, \textbf{202}(2015) 1069--1166.
	
	\bibitem[Yam19]{Yam19}
	\textsc{K.~Yamanoi}.
	\newblock Pseudo {K}obayashi hyperbolicity of subvarieties of general type on
	abelian varieties.
	\newblock \emph{J. Math. Soc. Japan}, \textbf{71}(2019) 259--298.
	
\end{thebibliography}
\end{document}